\documentclass[12pt]{amsart}

\usepackage{mathtools,amsthm,amssymb,tikz-cd,mathrsfs,scalerel}
\usepackage[colorlinks,linkcolor=red!75!black,citecolor=green!50!black,urlcolor=cyan!75!black]{hyperref}

\newcommand{\Bibkeyhack}[3]{}

\usepackage{lmodern}
\usepackage{microtype}

\usepackage[margin=1in]{geometry}
\renewcommand{\footnoterule}{\vfill\kern -3pt \hrule width 0.4\columnwidth \kern 2.6pt}

\usepackage{enumitem}
\setlist{topsep=0.5ex,itemsep=0ex,parsep=0ex,leftmargin=3em,listparindent=\parindent}

\newenvironment{conditions}[1][]{\enumerate[label=\textnormal{({\roman{enumi}})}]}{\endenumerate}

\newtheoremstyle{thmstyle}{}{}{\itshape}{}{\bfseries}{. }{0pt}{\thmnumber{#2.\ }\thmname{#1}\thmnote{\textnormal{ (#3)}}}

\newtheoremstyle{defnstyle}{}{}{}{}{\bfseries}{.}{ }{\thmnumber{#2.\ }\thmname{#1}\thmnote{\textnormal{ (#3)}}}

\newtheoremstyle{thingstyle}{}{}{}{}{\bfseries}{. }{0pt}{}

\theoremstyle{thmstyle}
\newtheorem{thm}{Theorem}
\newtheorem{lem}[thm]{Lemma}
\newtheorem{prop}[thm]{Proposition}
\newtheorem{cor}[thm]{Corollary}

\theoremstyle{defnstyle}
\newtheorem*{acknowledgements}{Acknowledgements}
\newtheorem*{ai}{Statement of AI use}

\newtheorem{defn}[thm]{Definition}
\newtheorem{eg}[thm]{Example}

\newtheorem*{notation}{Notation}
\newtheorem{ontheproofs}[thm]{On the proofs}

\newtheorem{rmk}[thm]{Remark}
\newtheorem*{results}{Summary of main results}

\theoremstyle{thingstyle}
\newtheorem{thing}[thm]{}

\counterwithin{equation}{thm}

\DeclareMathOperator{\charac}{char}
\DeclareMathOperator{\codim}{codim}

\DeclareMathOperator{\tr}{tr}

\newcommand{\Coker}{\mathrm{Coker}}
\newcommand{\End}{\mathrm{End}}
\newcommand{\Gal}{\mathrm{Gal}}

\newcommand{\Hom}{\mathrm{Hom}}
\newcommand{\id}{\mathrm{id}}
\newcommand{\Img}{\mathrm{Img}}
\newcommand{\Ker}{\mathrm{Ker}}

\newcommand{\Spec}{\mathrm{Spec}}

\newcommand{\mbb}{\mathbb}
\newcommand{\mbf}{\mathbf}
\newcommand{\mc}{\mathcal}
\newcommand{\mf}{\mathfrak}
\newcommand{\mr}{\mathrm}

\newcommand{\tit}{\textit}
\newcommand{\tn}{\textnormal}

\newcommand{\ol}{\overline}

\newcommand{\C}{\mathbb{C}}
\newcommand{\F}{\mathbb{F}}

\newcommand{\Gm}{\mathbb{G}_\mathrm{m}}
\newcommand{\cO}{\mathcal{O}}

\newcommand{\Q}{\mathbb{Q}}

\newcommand{\Z}{\mathbb{Z}}

\usepackage[only,llbracket,rrbracket]{stmaryrd}

\newcommand{\ceq}{\coloneqq}
\newcommand{\cln}{\colon}
\newcommand{\gen}[1]{\langle{#1}\rangle}
\newcommand{\into}{\hookrightarrow}
\newcommand{\iso}{\overset{\sim}{\to}}
\newcommand{\onto}{\mathrel{\mathrlap{\rightarrow}\mkern-2.25mu\rightarrow}}
\newcommand{\oo}{\infty}

\let\tempv\phi
\let\phi\varphi
\let\varphi\tempv

\let\tempe\epsilon
\let\epsilon\varepsilon
\let\varepsilon\tempe

\let\setminus\smallsetminus

\makeatletter
\newcommand{\filcolim@}[2]{%
  \vtop{\m@th\ialign{##\cr
    \hfil$#1\operator@font colim$\hfil\cr
    \noalign{\nointerlineskip\kern1.5\ex@}#2\cr
    \noalign{\nointerlineskip\kern-\ex@}\cr}}%
}
\newcommand{\filcolim}{%
  \mathop{\mathpalette\filcolim@{\rightarrowfill@\textstyle}}\nmlimits@
}
\makeatother

\newcommand{\Br}{\mathrm{Br}}
\newcommand{\CH}{\mathrm{CH}}
\newcommand{\CHM}{\mathbf{CHM}}
\newcommand{\cl}{\mathrm{cl}}

\newcommand{\DM}{\mathbf{DM}}
\newcommand{\DMgm}{\mathbf{DM}_\mathrm{gm}}

\newcommand{\fh}{\mathfrak{h}}

\newcommand{\ft}{\mathfrak{t}}
\newcommand{\ggen}[1]{\langle\!\langle#1\rangle\!\rangle}
\newcommand{\Kum}{\mathrm{Kum}}
\newcommand{\mot}{\mathrm{mot}}
\newcommand{\NS}{\mathrm{NS}}

\newcommand{\Pic}{\mathrm{Pic}}

\newcommand{\tors}{\mathrm{tors}}
\newcommand{\ur}{\mathrm{ur}}

\numberwithin{thm}{section}

\title{On the torsion in the Chow motive of an Enriques surface}
\author{Jake Huryn and William C.\ Newman}
\address{The Ohio State University}
\email{huryn.5@osu.edu}
\address{The Ohio State University}
\email{newman.801@osu.edu}
\date{}

\begin{document}

\begin{abstract}
The integral Chow motive of an Enriques surface $S$ is determined by a torsion summand $\ft(S)$, which contributes the torsion in the cohomology of $S$. We study isomorphisms between and endomorphisms of motives of the form $\ft(S)$. This is closely related to the integral Hodge conjecture on products of two Enriques surfaces.
\end{abstract}

\maketitle

\tableofcontents

\section*{Introduction}

\noindent
The structure of the category of Chow motives with $\Z$-coefficients is extremely rich and displays phenomena not present in the more frequently studied $\Q$-linear category.
One such phenomenon, which has been known for a while, is the existence of nontrivial torsion objects, that is, objects $M$ for which $(\End(M),+)$ is a torsion Abelian group.
The first examples were discovered by Gorchinskiy--Orlov (\cite[Proposition 2.3]{phantom}), who showed that for a smooth projective surface $S$ over $\C$ satisfying $p=q_g=0$ as well as $\CH_0(S)\cong\Z$ (the conclusion of Bloch's conjecture (\cite{bloch})), there is a decomposition of the Chow motive with $\Z$-coefficients
\[
\fh(S)
=
\mbf1\oplus\mbf1(-1)^{\rho(S)}\oplus\mbf1(-2)
\oplus
\ft(S)
\]
for some torsion Chow motive $\ft(S)$, where $\rho(S)$ is the Picard number of $S$.
Moreover, using Merkurjev--Suslin's theorem (\cite{ms}), they were able to compute that the torsion order of $(\End(\ft(S)),+)$ is equal to that of $\Pic(S)_\tors$ (see also \cite{kahn}).

There now arise several basic questions about the structure of these motives $\ft(S)$ and their relationships to each other; for example:
\begin{enumerate}
\item
\label{item:qDEC}
When is $\ft(S)$ decomposable or indecomposable?
\item
\label{item:qISO}
When are $\ft(S)$ and $\ft(S')$ isomorphic?
\item
\label{item:qHOM}
Can we compute the Abelian group $\Hom(\ft(S),\ft(S'))$?
\end{enumerate}
These questions were, to some extent, addressed by Vishik in the paper \cite{vishik}, which contains a detailed study of torsion motives and numerous interesting results about them.
In particular, Vishik showed that $p$-torsion direct summands of the motives of surfaces are classified, up to isomorphism, by Rost cycle modules of a specific form (\cite[Theorem 4.13]{vishik}), and used this to obtain information about the endomorphism rings of such motives (\cite[Proposition 4.24 and Remark 4.25]{vishik}).
But it remained unclear how to answer the questions (a)--(c) in any specific concrete setting (aside from trivial cases, such as when $\Pic(S)_\tors\ncong\Pic(S')_\tors$).
The \tit{raison d'\^etre} of this paper is to remedy this situation and bring further attention to this interesting class of motives.

An ``Enriques surface'' over a field $k$ is a minimal smooth projective surface $S$ over $k$ of Kodaira dimension $0$ and second Betti number $b_2(S)=10$.
These are known to satisfy the hypotheses of \cite[Proposition 2.3]{phantom} when $k=\C$, so $\ft(S)$ exists in this case.
In fact, as noticed in \cite[\S3]{sy}, the same argument produces a decomposition $\fh(S)_{\Z[1/p]}=\mbf1\oplus\mbf1(-1)^{10}\oplus\mbf1(-2)\oplus\ft(S)$, where $\fh(S)_{\Z[1/p]}$ denotes the Chow motive of $S$ with $\Z[1/p]$-coefficients, for any algebraically closed field $k$ of exponential characteristic $p$ (that is, $p\ceq1$ if $\charac(k)=0$, and otherwise $p\ceq\charac(k)$).
Enriques surfaces are natural to consider because, from the perspective of Kodaira dimension, they are the simplest surfaces admitting such a decomposition in which $\ft(S)$ is nontrivial.

\begin{results}
We give partial answers to the questions \ref{item:qDEC}--\ref{item:qHOM} above;
see \S\ref{sec:results} for detailed statements and an overview of their proofs.
Briefly, over any algebraically closed field of characteristic not $2$ other than the algebraic closure of a finite field, we
\begin{enumerate}
\item
show that $\ft(S)$ may be decomposable.
\item
show that $\ft(S)\cong\ft(S')$ for ``many'' pairs of non-isomorphic Enriques surfaces $(S,S')$, while $\ft(S)\ncong\ft(S')$ for ``many'' pairs $(S,S')$, even some having the same K3 cover.
\item
completely determine $\Hom(\ft(S),\ft(S'))$ for ``many'' $(S,S')$.
\end{enumerate}
Moreover, nearly all our examples of these phenomena are completely explicit.
\end{results}

\noindent
We now compare our results with some results in literature.
First, the fact that $\ft(S)$ may be decomposable when $k$ is of characteristic $0$ directly contradicts \cite[Proposition 4.21]{vishik}.
In Appendix \ref{appx:vishik}, we explain what goes wrong in the argument of \textit{loc.\ cit}.
In turn, the statements in \cite[\S7.3]{sy} are also wrong (see Remark \ref{rmk:sycomp} below): for such $S$, we have $\CH_0(S_{k(S)})=\Z/2\oplus\Z/2$ and $H_\mr{ur}^3(S\times S,\Q_\ell/\Z_\ell(2))=0$ for any prime $\ell\neq\charac(k)$ (answering \cite[footnote 1]{kahn}), and if $k=\C$, $S\times S$ satisfies the integral Hodge conjecture in degree $4$.
In fact, we have been unable to find a single example of an Enriques surface $S$ for which $\ft(S)$ is indecomposable, or equivalently if $k=\C$, for which $S\times S$ fails the integral Hodge conjecture in degree $4$.
Nevertheless, we believe that $\ft(S)$ should ``usually'' be indecomposable.

Second, if $k=\C$, that $\ft(S)\ncong\ft(S')$ implies $S\times S'$ fails the integral Hodge conjecture in degree $4$.
Such failure was first demonstrated by Diaz in \cite{diaz}, based on arguments of Colliot-Th\'el\`ene--Voisin (\cite{ct-v}) and Gabber (\cite[Appendice]{ct-g}; see also \cite{ct}), and in fact it is not hard to deduce the statement that $\ft(S)\ncong\ft(S')$ directly from Diaz's argument.
However, one of the aims of this paper is to find examples which are as explicit and numerous as possible, and we succeed in extending Diaz's argument to apply in more situations.
This relies on a computation of the Brauer pullback $\Br(S)\to\Br(X)$, where $X\to S$ is the K3 cover, and to this end, we give a complete description of this map (Corollary \ref{cor:enriques-brauer-pullback}) for all $S$ in a certain $2$-dimensional family of Enriques surfaces, namely those defined by a Lieberman involution (see \eqref{thing:lieberman}).

\begin{acknowledgements}
We are very grateful to Alexander Vishik for discussing the material of \cite{vishik} with us and for providing feedback on an early draft of this paper.
We also thank Igor Dolgachev and Daniel Huybrechts for helpful conversations.
This material is based upon work supported by the National Science Foundation under
Grant No.\ DMS-2231565.
\end{acknowledgements}

\begin{ai}
LLMs were used to find typos, check for mathematical accuracy, search for references, and to better understand some of the background material.
The initial draft of this paper was generated entirely by the authors.
\end{ai}

\begin{notation}
Throughout this paper, we will observe the following conventions.
\begin{itemize}
\item
We denote by $k$ an algebraically closed field and $p$ the exponential characteristic of $k$.
\item
If $E$ is an elliptic curve over $k$ and $n$ an integer not divisible by $\charac(k)$, we will very often identify the $k$-group scheme $E[n]$ with its set of $k$-points.
\item
All sheaves are \'etale sheaves, and $H^*$ and $\pi_1$ always stand for \'etale cohomology and the \'etale fundamental group, respectively.
(Of course, for a field $K$, we will sometimes find it useful to identify $\pi_1(\Spec(K))$ with the absolute Galois group $\Gal(\ol K/K)$ and $H^*(\Spec(K))$ with the group cohomology of $\Gal(\ol K/K)$.
\item
For a scheme $X$, we write $\Br(X)\ceq H^2(X,\mbb G_\mr{m})_\tors$ for the cohomological Brauer group of $X$.
\item
We will, in certain situations, be careful to distinguish between Tate twists, mainly cohomology with $\Z/2$-coefficients versus $\mu_2$-coefficients.
Since $\Z/2$ and $\mu_2$ are uniquely isomorphic over a field of characteristic not $2$, the difference is often immaterial.
However, one situation in which it is necessary (to obtain the correct statement of Theorem \ref{thm:enriques-thm-2}) for us to have a precise understanding of the relationship between $\Z/2$- and $\mu_2$-cohomology is explained in \eqref{thing:ec-coh}.
\end{itemize}
\end{notation}

\section{Statements of main results}
\label{sec:results}

\noindent
In this section, $k$ is an algebraically closed field of characteristic not $2$.
Our main theorems deal with the following class of Enriques surfaces.

\begin{thing}
\label{thing:lieberman}
Let $\mf L(k)$ denote the set of isomorphism classes of tuples $(E,P,F,Q)$ in which $E$ and $F$ are elliptic curves over $k$ and $P\in E[2]\setminus\{O\}$ and $Q\in F[2]\setminus\{O\}$.
To ease notation, if we write something like ``$\circ\in\mf L$'', then $(E_\circ,P_\circ,F_\circ,Q_\circ)$ will denote the tuple corresponding to $\circ$.
Now fix $\circ\in\mf L(k)$.
Recall that the Kummer surface $\Kum(E_\circ\times F_\circ)$ is the minimal resolution of the quotient of $(E_\circ\times F_\circ)/\iota$, where $\iota$ multiplication by $-1$; alternatively, $\iota$ lifts uniquely to an involution of the blowup $\mr{Bl}_Z(E_\circ\times F_\circ)$ along the $\iota$-fixed locus $Z\ceq E_\circ[2]\times F_\circ[2]$, and $\Kum(E_\circ\times F_\circ)$ is the quotient $\mr{Bl}_Z(E_\circ\times F_\circ)/\iota$.
The involution of $E_\circ\times F_\circ$ defined by $(x,y)\mapsto(x+P_\circ,-y+Q_\circ)$ descends to a fixed-point--free involution $\sigma_\circ$ of $\Kum(E_\circ\times F_\circ)$, a ``Lieberman involution'', and we denote by $S_\circ\ceq\Kum(E_\circ\times F_\circ)/\sigma_\circ$ the corresponding Enriques surface.
We will sometimes say that such an Enriques surface is of ``Lieberman type''.
\end{thing}

\begin{thm}
\label{thm:enriques-thm-1}
Let $k$ be an algebraical{}ly closed field of characteristic not $2$.
\begin{enumerate}
\item\label{item:et1ISO}
Let $\circ,\bullet\in\mf L(k)$, and assume there exist isogenies $E_\circ\to E_\bullet$ and $F_\circ\to F_\bullet$ of odd degree sending $P_\circ$ to $P_\bullet$ and $Q_\circ$ to $Q_\bullet$, respectively.
Then $\ft(S_\circ)\cong\ft(S_\bullet)$.
\item\label{item:et1DECOMP}
Let $E_\circ$ be an el{}liptic curve over $k$ such that $\End(E_\circ)\cong\mc O_L$ for some quadratic imaginary number field $L$ in which $2$ splits.
Then there exists $P_\circ\in E_\circ[2]\setminus\{O\}$ such that $\ft(S_\circ)$ is decomposable for any $\circ\in\mf L(k)$ of the form $\circ=(E_\circ,P_\circ,F_\circ,Q_\circ)$.
\end{enumerate}
\end{thm}

\begin{rmk}
In order for \ref{item:et1ISO} of Theorem \ref{thm:enriques-thm-1} to be interesting, we need to explain why $S_\circ$ and $S_\bullet$ can be chosen to be not isomorphic as varieties.
When $k=\C$, this follows from the general fact that for a K3 surface $X$, there are only finitely many Abelian varieties $A$ satisfying $X\cong\Kum(A)$ by \cite[Theorem 0.1(1)]{kummer-structures}.
Combining \tit{loc.\ cit.\ }with \cite[Theorem 2.19]{orlov-02} and a short argument, one can show that if $k=\C$, $E_\circ$ and $F_\circ$ are not isogenous, $\End(E_\circ)\cong\Z$,\footnote{
or, more generally, $\End(E_\circ)$ is isomorphic to an imaginary quadratic order of odd class number
}
and either $E_\circ\ncong E_\bullet$ or $E_\circ\ncong F_\bullet$, then $S_\circ\ncong S_\bullet$.
Over a general algebraically closed field of characteristic not $2$, one can use the facts that any Enriques surface admits at most 527 elliptic fibrations up to isomorphism (\cite[Theorem 5.4]{bg}) and $S_\circ$ admits two elliptic fibrations with respective half-fiber isomorphic to $E_\circ/P_\circ$ or $F_\circ/Q_\circ$.
\end{rmk}

\begin{rmk}
More can be said about the decomposition of \ref{thm:enriques-thm-1}\ref{item:et1DECOMP}; see part \ref{item:cDEC} of Lemma \ref{lem:c}.
In particular, it yields the first example of a nontrivial Chow motive $M$ over $\C$ with $\Z$-coefficients having trivial total Chow group, i.e.\ $\CH^i(M)=0$ for all $i$.
By contrast, a nontrivial Chow motive over $\C$ with $\mathbb Q$-coefficients has nontrivial total Chow group (\cite[Lemma 2.1]{gg}), and therefore any example of such a motive $M$ must be torsion.
\end{rmk}

\noindent
In order to state our second main theorem, we will need the following rather unwieldy notation and terminology.

\begin{thing}
\label{thing:legendre}
Given an elliptic curve $E$ over $k$ and $P\in E[2]\setminus\{O\}$, there exists $\lambda\in k\setminus\{0,1\}$ for which $E$ is isomorphic to the curve $E_\lambda$ given by the Weierstrass equation $y^2=x(x-1)(x-\lambda)$ and $P$, under this isomorphism, is sent to the point $P_\lambda\ceq(0,0)$ on $E_\lambda$ (see \cite[Ch.\ III, Proposition 1.7]{silverman}).
There exists exactly one other value with this property, namely $1/\lambda$.
So, for $\lambda,\mu\in k\setminus\{0,1\}$, we will denote by $S_{\lambda,\mu}$ the Enriques surface attached to $(E_\lambda,P_\lambda,E_\mu,P_\mu)\in\mf L(k)$ as in \eqref{thing:lieberman}.
\end{thing}

\begin{thing}
\label{thing:star}
Let $\lambda\in k\setminus\{0,1\}$, and let $P$ be the nonzero element in the image of $E_\lambda[2]\to(E_\lambda/P_\lambda)[2]$.
A direct calculation (using, for example, \cite[Ch.\ III, Example 4.5]{silverman}) gives $(E_\lambda/P_\lambda,P)\cong(E_{\lambda^*},P_{\lambda^*})$, where
$
\lambda^*
\ceq
(\sqrt\lambda+1)^2/(\sqrt\lambda-1)^2
$.
\end{thing}

\begin{defn}
\label{defn:suitable}
Given $\lambda,\mu,\nu\in k\setminus\{0,1\}$, let us say that $(\lambda,\mu,\nu)$ is \tit{suitable}\footnote{
by which we mean ``suitable for the proof of Corollary \ref{cor:galois-nonzero-application}''
}
if there exists a subfield $k_0$ of $k$ and a discrete valuation $v$ on $k_0$ with residue characteristic not $2$ such that $\lambda,\mu,\nu\in k_0$ and at least one of the fol{}lowing holds:
\begin{enumerate}
\item\label{item:gnaGGB}
$v(\lambda)=v(\lambda-1)=v(\mu)=v(\mu-1)=0$ and $v(\nu)\neq0$.
\item\label{item:gnaBB}
For some indexing $\{\lambda,\mu,\nu\}=\{\lambda_1,\lambda_2,\lambda_3\}$, we have
$v(\lambda_1-1)=0$, $v(\lambda_2)\neq0$, and $v(\lambda_3)\neq0$.
\item\label{item:gnaB}
For some indexing $\{\lambda,\mu,\nu\}=\{\lambda_1,\lambda_2,\lambda_3\}$, we have $v(\lambda)=v(\lambda-1)=v(\mu)=v(\mu-1)=0$, $v(\lambda_3)\neq0$, and moreover there is no isogeny $\phi\cln E_{\ol{\lambda_1}}\to E_{\ol{\lambda_2}}$ such that $\phi(P_{\ol{\lambda_1}})=0$ and $P_{\ol{\lambda_2}}\in\phi(E_{\ol{\lambda_1}}[2])$, where the bar denotes the image in a fixed algebraic closure of the residue field of $v$.
\end{enumerate}
Note that $(\lambda,\mu,\nu)$ is \tit{never} suitable if $k$ is the algebraic closure of a finite field.
\end{defn}

\begin{thm}
\label{thm:enriques-thm-2}
Let $k$ be an algebraical{}ly closed field of characteristic not $2$, and fix $\lambda,\mu,\lambda',\mu'\in k\setminus\{0,1\}$.
Assume that one of $(\lambda,\mu,\lambda'^*)$, $(\lambda,\mu,\mu'^*)$, $(\lambda',\mu',\lambda^*)$, or $(\lambda',\mu',\mu^*)$ is suitable.
Then $\ft(S_{\lambda,\mu})\ncong\ft(S_{\lambda',\mu'})$.
\end{thm}

\begin{rmk}
It is easy to use Theorem \ref{thm:enriques-thm-2} to find Enriques surfaces $S$ and $S'$ having the same K3 cover but satisfying $\ft(S)\ncong\ft(S')$; see Example \ref{eg:same-cover}.
\end{rmk}

\begin{rmk}
In several cases, we are able to exactly compute $\Hom(\ft(S),\ft(S'))$ (Examples \ref{eg:c-possibilities} and \ref{eg:c-possibilities-2}).
In fact, we show that four out of the five possibilities of this group (Lemma \ref{lem:c}) do occur; the missing possibility occurs if and only if some $\ft(S)$ is indecomposable.
\end{rmk}

\begin{ontheproofs}
After some preliminaries in \S\S\ref{sec:brauer}--\ref{sec:surface-motives}, the proofs of Theorems \ref{thm:enriques-thm-1} and \ref{thm:enriques-thm-2} are completed in \S\ref{sec:enriques-motives}.
The proofs are quite different in nature:
The former requires us to produce algebraic cycles, while the latter requires us to show that a certain cohomology class is not algebraic.

The proof of Theorem \ref{thm:enriques-thm-1} rests on a finiteness result, in the spirit of Merkurjev--Suslin (\cite[\S18]{ms}), for torsion in $\CH^2$ (Proposition \ref{prop:bloch-kato}).
This result puts strong restrictions on morphisms and endomorphisms of torsion summands of the motives of smooth projective surfaces (\eqref{thing:tis-hom}--\eqref{lem:action-on-coh}).
In particular, we observe that Theorem \ref{thm:enriques-thm-1} is implied by the algebraicity of certain $2$-torsion cohomology classes in $H^4(S_\circ\times S_\bullet,\Z_2(2))$ (Lemma \ref{lem:c}).
A simple geometric argument shows that rational maps of odd or even degree provide the necessary algebraic cycles (Proposition \ref{prop:enriques-rat-map}).
Finally, we construct the rational maps in \S\ref{sec:enriques-motives}.

The idea of the proof of Theorem \ref{thm:enriques-thm-2} is to study $\Hom(\ft(S),\fh(E))$ for Enriques surfaces $S$ and elliptic curves $E$.
More precisely, since $H^1(\ft(S),\mu_2)\cong\Z/2$, we get a map $\Hom(\ft(S),\fh(E))\to H^1(E,\mu_2)$, and the image of this map is an invariant of $\ft(S)$.
Assuming, for example, that $(\lambda,\mu,\lambda'^*)$ is suitable, this reduces the claim to the non-algebraicity of a certain $2$-torsion class in $H^4(S_{\lambda,\mu}\times E_{\lambda'^*},\Z_2(2))$.
By the work of Colliot-Th\'el\`ene--Voisin (\cite[Th\'eor\`eme 3.7]{ct-v}; see Proposition \ref{prop:ur}), it suffices to show that a certain class in $H^3(S_{\lambda,\mu}\times E_{\lambda'^*},\mu_2^{\otimes2})$ is nonvanishing in Galois cohomology.
For this, we use modified version of a degeneration argument of Gabber (\cite[Appendice]{ct-g}, \cite[Theorem 3.3]{auel-suresh}; see Proposition \ref{prop:galois-nonzero}), supplemented by the explicit calculation of the Brauer pullback for Enriques surfaces of Lieberman type (Proposition \ref{prop:enriques-brauer-pullback}).
We remark that the integral Hodge conjecture for varieties of the form $S\times E$ was first shown to be false by Benoist--Ottem in \cite[Theorem 0.1]{bo} using a Hilbert-scheme argument.
\end{ontheproofs}

\begin{rmk}
Aside from finding an example when $\ft(S)$ is indecomposable, the other major shortcoming of our results is that we are unable to find a single pair of Enriques surfaces $S$ and $S'$ over the algebraic closure of a finite field such that $\ft(S)\ncong\ft(S')$, because the argument of Proposition \ref{prop:galois-nonzero} relies on passing to the residue field of a (nontrivial) valuation on the base field.
What is more, Schoen proved that over such a field, the integral Tate conjecture for $1$-cycles is a consequence of the usual Tate conjecture for divisors on surfaces (\cite{schoen}).
Thus we expect that the map $\Hom(\ft(S),\fh(E))\to H^1(E,\mu_2)$ is surjective for all $S$ and $E$, rendering the motivic invariants $\Img(\Hom(\ft(S),\fh(E)\to H^1(E,\mu_2))$ completely useless.
\end{rmk}

\section{Brauer classes on Enriques surfaces of Lieberman type}
\label{sec:brauer}

\noindent
In this section, $k$ is an algebraically closed field of characteristic not $2$.
If $S$ is an Enriques surface over $k$, then $\Br(S)\cong\Z/2$ (\cite[Theorem 1.2.17]{enriques-I}).
In \cite{beauville}, Beauville gives an explicit lattice-theoretic description of the locus of complex Enriques surfaces $S$ for which the map $\Br(S)\to\Br(X)$ vanishes, where $X\to S$ is the K3 cover.
In subsequent work, such as \cite{gs}, explicit families of Enriques surfaces with vanishing Brauer pullback were exhibited.
The purpose of this section is to compute $\Br(S)\to\Br(X)$ for any Enriques surface $S$ of Lieberman type over $k$.
Note that the Enriques surfaces in \cite{gs} are also constructed using products of elliptic curves; we have not tried to determine the relationship between them and ones of Lieberman type.

\begin{thing}
\label{thing:enriques-brauer}
Let $S$ be an Enriques surface over $k$.
Then, from the $\mu_2$-Kummer sequence and the fact that $\Br(S)\cong\Z/2$, we get a surjective map $H^2(S,\mu_2)\onto\Br(S)$ whose kernel is the group of algebraic classes.
A comparison of the $\mu_2$- and $\mu_4$-Kummer sequences shows that a class in $H^2(S,\mu_2)$ is algebraic if and only if it lifts to a class in $H^2(S,\mu_4)$.
\end{thing}

\begin{thing}
\label{thing:blowup}
First, a reminder on the cohomology of a blowup.
Let $\iota\cln Z\into X$ be a codimension-$c$ closed immersion of smooth $k$-varieties, let $\pi\cln\mr{Bl}_Z(X)\to X$ denote the blowup, and let $\iota'\cln Z'\into\mr{Bl}_Z(X)$ denote the inclusion of the exceptional divisor.
Then the map
\[
\begin{pmatrix}
\pi^*&\iota'_*\\
0&(\pi|_{Z'})_*
\end{pmatrix}
\cln
H^i(X,\mu_2^{\otimes j})
\oplus
H^{i-2}(Z',\mu_2^{\otimes j-1})
\to
H^i(\mr{Bl}_Z(X),\mu_2^{\otimes j})
\oplus
H^{i-2c}(Z,\mu_2^{\otimes j-c})
\]
is an isomorphism (\cite[Expos\'e VI, Corollaire 8.3]{sga5}).
This is the ``blowup formula''.
\end{thing}

\begin{thing}
\label{thing:lieberman-setup}
We will use the notation and terminology of \eqref{thing:lieberman}.
Fix $(E,P,F,Q)\in\mf L(k)$.
Set $Z\ceq E[2]\times F[2]$ and $Y^\circ\ceq(E\times F)\setminus Z$, and let $Y$ denote the blowup of $E\times F$ along $Z$.
Set $X\ceq\Kum(E\times F)$, and let $X\to S$ be the quotient by the Lieberman involution attached to $(E,P,F,Q)$.
We get a commutative diagram
\[
\begin{tikzcd}
[column sep=small,row sep=small]
H^2(S,\mu_2)
\arrow[r]\arrow[d]&
H^2(X,\mu_2)
\arrow[r]\arrow[d]&
H^2(Y,\mu_2)
\arrow[r]\arrow[d]&
H^2(Y^\circ,\mu_2)&
H^2(E\times F,\mu_2)
\arrow[l,"\sim"']\arrow[d]
\\
\Br(S)
\arrow[r]&
\Br(X)
\arrow[r]&
\Br(Y)
&&
\Br(E\times F);
\arrow[ll,"\sim"']
\end{tikzcd}
\]
the last arrow on the top is an isomorphism by the Gysin sequence (\cite[Ch.\ VI, Remark 5.4(b)]{milne-book}) since $\codim_{E\times F}(Z)=2$, and the last arrow on the bottom can be checked to be an isomorphism using the Kummer sequence and the blowup formula \eqref{thing:blowup}.
\end{thing}

\begin{thing}
\label{thing:ec-coh}
Let $E$ be an elliptic curve over $k$.
Recall that $H^1(E,\Z/2)$ is isomorphic to the group of \'etale $\Z/2$-torsors over $E$.
\begin{enumerate}
\item
\label{item:ecISO}
We will frequently make use of the isomorphism $\Hom(E[2],\Z/2)\iso H^1(E,\Z/2)$, which we will denote $\theta\mapsto\alpha_\theta$, characterized as follows:
For a surjection $\theta\cln E[2]\onto\Z/2$, $\alpha_\theta$ is given by the isogeny $E/{\Ker(\theta)}\to E$ dual to the quotient isogeny.
In turn, using the perfect Poincar\'e pairing $H^1(E,\Z/2)\times H^1(E,\mu_2)\to\Z/2$, we get an isomorphism $E[2]\iso H^1(E,\mu_2)$, which we will also denote by $P\mapsto\alpha_P$.
\item
\label{item:ec2}
Given $P\in E[2]\setminus\{O\}$, there is a unique homomorphism $P^\vee\cln E[2]\iso\Z/2$ with kernel $\{O,P\}$.
Using that the pairing on $H^1(E,\mu_2)$ is alternating, one checks that the isomorphism $H^1(E,\Z/2)\iso H^1(E,\mu_2)$ induced by $\Z/2\iso\mu_2$ sends $\alpha_{P^\vee}$ to $\alpha_P$ for each $P\in E[2]\setminus\{O\}$.
\end{enumerate}
\end{thing}

\begin{prop}
\label{prop:enriques-brauer-pullback}
In the setting of \tn(\ref{thing:lieberman-setup}\tn), let $\beta\in H^2(S,\mu_2)$ be any non-algebraic class.
Then the image of $\beta$ under
\[
H^2(S,\mu_2)
\to
H^2(E\times F,\mu_2)
\to
H^1(E,\Z/2)\otimes H^1(F,\mu_2)
\]
equals $\alpha_{P^\vee}\otimes\alpha_Q$ in the notation of \ref{thing:ec-coh}\ref{item:ecISO}.
\end{prop}

\noindent
Before giving the proof of Proposition \ref{prop:enriques-brauer-pullback}, we provide some reminders on \'etale local systems.

\begin{thing}
\label{thing:lcc1}
Let $T$ be a scheme.
An \'etale sheaf on $T$ is called \tit{lcc} if it is locally constant and has finite stalks.
For such a sheaf $\mathscr F$, the stalk $\mathscr F_{\ol t}$ at a geometric point $\ol t$ has a ``monodromy'' action of $\pi_1(T,\ol t)$ which is functorial in $(T,\ol t)$.
If $T$ is connected, $\mathscr F\mapsto\mathscr F_{\ol t}$ defines an equivalence of categories (\cite[Section \href{https://stacks.math.columbia.edu/tag/0DV4}{\texttt{0DV4}}]{stacks})
\[
(\text{lcc sheaves of sets on $T$})
\iso
(
\text{continuous actions of $\pi_1(T,\ol t)$ on finite sets}
).
\]
Under this equivalence, constant sheaves correspond to $\pi_1$-sets with the trivial action.
More generally, the image of the (injective) map $\Gamma(T,\mathscr F)\to\mathscr F_{\ol t}$ is the set of $\pi_1$-invariants.
If $T=T_1\times_\Sigma T_2$ and $\mathscr F_i$ is an lcc Abelian sheaf on $T_i$, we denote by $\mathscr F_1\boxtimes\mathscr F_2$ the ``external product'' sheaf $\mr{pr}_1^{-1}\mathscr F_1\otimes\mr{pr}_2^{-1}\mathscr F_2$ on $T$, which is again lcc.
If $\ol t=(\ol t_1,\ol t_2)$, then $(\mathscr F_1\boxtimes\mathscr F_2)_{\ol t}\cong\mathscr F_{1,\ol t_1}\otimes\mathscr F_{2,\ol t_2}$.
\end{thing}

\begin{thing}
\label{thing:rif*}
Let $T$ be a scheme.
For a smooth proper morphism $f\cln\mc X\to T$ and an lcc Abelian sheaf $\mathscr F$ on $\mc X$ whose stalks have order invertible on $T$, we will write $\mathscr H^i(\mc X/T,\mathscr F)\ceq R^if_*\mathscr F$.
By the theorems on smooth and proper base change (\cite[Ch.\ VI, Corollary 4.2]{milne-book}), this is an lcc Abelian sheaf on $T$, and at a geometric point $\ol t$, its stalk $\mathscr H^i(\mc X/T,\mathscr F)_{\ol t}$ is naturally isomorphic to $H^i(\mc X_{\ol t},\iota^{-1}\mathscr F)$, where $\iota$ denotes the inclusion $\mc X_{\ol t}\into\mc X$.
If $T=T_1\times_\Sigma T_2$, $\mc X=\mc X_1\times_\Sigma\mc X_2$ for smooth proper $T_i$-schemes $\mc X_i$, and $\mathscr F=\mathscr F_1\boxtimes\mathscr F_2$ for lcc Abelian sheaves $\mathscr F_i$ on $\mc X_i$ of order invertible on $T_i$, the K\"unneth map
\[
\bigoplus_{a+b=i}
\mathscr H^a(\mc X_1/T_1,\mathscr F_1)
\boxtimes
\mathscr H^b(\mc X_2/T_2,\mathscr F_2)
\to
\mathscr H^i(\mc X/T,\mathscr F)
\]
is an isomorphism of sheaves, because it is an isomorphism on stalks by the usual K\"unneth theorem (\cite[Ch.\ VI, Corollary 8.13]{milne-book}).
\end{thing}

\begin{proof}[Proof of Proposition \ref{prop:enriques-brauer-pullback}]
Let $\pi\cln Y\to E\times F$ denote the blowup morphism.
We first claim that it is equivalent to compute the image of $\beta$ under
\[
H^2(S,\mu_2)
\to
H^2(Y,\mu_2)
\xrightarrow{\pi_*}
H^2(E\times F,\mu_2)
\to
H^1(E,\Z/2)\otimes H^1(F,\mu_2).
\]
For this, we will show that the triangle
\[
\begin{tikzcd}
[column sep=0em,row sep=small]
H^2(Y,\mu_2)
\arrow[rr,"\pi_*"]
\arrow[rd,"u'^*"']
&&
H^2(E\times F,\mu_2)
\arrow[ld,"u^*"]
\\
&
H^2(Y^\circ,\mu_2)
\end{tikzcd}
\]
commutes, where $u$ and $u'$ are the respective open embeddings.
Let $\iota$ be the inclusion $Z\into E\times F$, and let $\iota'\cln Z'\into Y$ be the inclusion of the exceptional divisor.
Then we have $u'^*\circ \pi^*=u^*$ and $\pi_*\circ\iota_*'=\iota_*\circ(\pi|_{Z'})_*$ by functoriality, $u^*\circ\iota_*=0$ and $u'^*\circ\iota'_*=0$ by the Gysin sequence (\cite[Ch.\ VI, Remark 5.4(b)]{milne-book}; we also are using \cite[Ch.\ VI, Remark 11.6(b)]{milne-book}), and $\pi_*\circ\pi^*=\mr{id}$ because $\pi$ is birational.
Now fix $\xi\in H^2(Y,\mu_2)$.
By the blowup formula \eqref{thing:blowup}, we can (uniquely) write $\xi=\pi^*\xi_0+\iota'_*\zeta$ for some $\xi_0\in H^2(E\times F,\mu_2)$ and $\zeta\in H^0(Z',\mu_2)$.
Then
\[
u'^*\xi
=
u'^*\pi^*\xi_0+u'^*\iota'_*\zeta
=
u^*\xi_0
\]
and
\[
u^*\pi_*\xi
=
u^*\pi_*\pi^*\xi_0
+
u^*\pi_*\iota'_*\zeta
=
u^*\xi_0
+
u^*\iota_*(\pi|_{Z'})_*\zeta
=
u^*\xi_0,
\]
proving the commutativity of the triangle.

The open modular curve $Y_1(4)$, the fine moduli space of elliptic curves with a marked point of order $4$, exists as a connected scheme over $\Z[1/2]$ (\cite{mumford-book});\footnote{
It would be more natural to argue with $Y_1(2)$ instead, but this would require viewing $Y_1(2)$ as a stack; we prefer to use only schemes in this paper.
}
let $\mc E\to Y_1(4)$ denote its universal elliptic curve.
Set $T\ceq Y_1(4)\times_\Z Y_1(4)$, and let $\mc Y\to\mc E\times_\Z\mc E$ be the blowup along the $2$-torsion locus $\mc Z\ceq\mc E[2]\times_\Z\mc E[2]$.
Since both $\mc E\times_\Z\mc E$ and $\mc Z$ are smooth and proper over $T$, so is $\mc Y$, and for any $T'\to T$, the base-change $\mc Y_{T'}$ is naturally identified with the blowup of $(\mc E\times_\Z\mc E)_T$ along $\mc Z_T$.
Finally, we can define quotients $\mc Y\to\mc X\to\mc S$ by the Kummer and Lieberman involutions, respectively, where for the latter we use that $\mc E$ has a distinguished $2$-torsion section, namely twice the distinguished $4$-torsion section.
To summarize, $\mc Y\to\mc S$ is the morphism of smooth proper $T$-schemes whose fiber over $((E,P'),(F,Q'))\in T(k)$, where $2P'=P$ and $2Q'=Q$, is the map $Y\to S$ defined in \eqref{thing:lieberman-setup}.

By the first paragraph, we are interested in the composition
\begin{equation}
\label{eq:rpb}
\begin{aligned}
\mathscr H^2(\mc S/T,\mu_2)
\to
\mathscr H^2(\mc Y/T,\mu_2)
&\xrightarrow{\pi_*}
\mathscr H^2(\mc E\times_\Z\mc E/T,\mu_2)
\\&\hspace{3.5em}\to
\mathscr H^1(\mc E/Y_1(4),\Z/2)
\boxtimes
\mathscr H^1(\mc E/Y_1(4),\mu_2),
\end{aligned}
\end{equation}
where $\pi$ now denotes the $T$-morphism $\mc Y\to\mc E\times_\Z\mc E$.
Notice that $\mathscr H^2(\mc S/T,\mu_2)$ admits an lcc subsheaf $\mathscr G$ whose stalk at a geometric point $\ol t$ of $T$ is the space of algebraic classes in $H^2(\mc S_{\ol t},\mu_2)$; indeed, the image of $\mathscr H^2(\mc S/T,\mu_4)\to\mathscr H^2(\mc S/T,\mu_2)$ has this property, as noted in \eqref{thing:enriques-brauer}.
We claim that $\mathscr G$ is included in the kernel of \eqref{eq:rpb}.
Indeed, when $\ol t=(\ol t_1,\ol t_2)$ and the corresponding elliptic curves $\mc E_{\ol t_1}$ and $\mc E_{\ol t_2}$ are not isogenous, the only algebraic classes in $H^2(\mc E_{\ol t_1}\times\mc E_{\ol t_2},\mu_2)$ are classes in the $(0,2)$- and $(2,0)$-pieces of the K\"unneth decomposition.
Since every sheaf in sight is locally constant and $T$ is connected, inclusion at one stalk implies inclusion everywhere, proving the claim.
Therefore, \eqref{eq:rpb} induces a map
\begin{equation}
\label{eq:relbrpullback}
\mathscr H^2(\mc S/T,\mu_2)/\mathscr G
\to
\mathscr H
\ceq
\mathscr H^1(\mc E/Y_1(4),\Z/2)\boxtimes\mathscr H^1(\mc E/Y_1(4),\mu_2).
\end{equation}
But each stalk of $\mathscr H^2(\mc S/T,\mu_2)/\mathscr G$ is isomorphic to $\Br(S)[2]\cong\Z/2$, so this sheaf is constant (because $\Z/2$ has no nontrivial $\pi_1$-action), whence the image of \eqref{eq:relbrpullback} defines a global section of $\mathscr H$.
Thus we will be done if we can verify the following:
\begin{enumerate}
\item\label{item:ebpNZ}
The map \eqref{eq:relbrpullback} is nonzero.
\item\label{item:ebpGS}
The sheaf $\mathscr H$ has just one nonzero global section, and the stalk of this section over $((E,P'),(F,Q'))\in T(k)$ is the one described in the statement of the present proposition.
\end{enumerate}
To prove \ref{item:ebpNZ}, it suffices to check that $\Br(\mc S_{\ol t})\to\Br(\mc Y_{\ol t})$ is nonzero for some $\ol t$.
As a consequence of \cite[Corollary B(i)]{sv-brauer}, there exists $\ol t\in T(\C)$ for which the pullback map $\Br(\mc S_{\ol t})\to\Br(\mc X_{\ol t})$ is nonzero, where $\mc X_{\ol t}\to\mc S_{\ol t}$ is the K3 cover. 
Examining the Hochschild--Serre spectral sequence $E_2^{a,b}\ceq H^a(\Z/2,H^b(\mc Y_{\ol t}^\circ,\Gm))\Rightarrow H^{a+b}(\mc X_{\ol t}^\circ,\Gm)$, one sees that the pullback map $\Br(\mc X_{\ol t})\to\Br(\mc Y_{\ol t})$ is an isomorphism (see \cite[Proposition 1.3]{sz-brauer}).
Thus \ref{item:ebpNZ} holds.

Now we prove \ref{item:ebpGS}.
It is clear that $\mathscr H$ has a global section with the stated description.
There are several ways to see that it is the only nonzero global section; here is one of them.
Consider the elliptic curves $E_0$ and $F_0$ over $\Q$ given, respectively, by the equations
\[
y^2=x(x^2+4)
\quad\text{and}\quad
y^2=x(x^2-x+1).
\]
We have $E_0[4](\Q)\cong F_0[4](\Q)\cong\Z/4$ (see \cite[elliptic curves \href{https://www.lmfdb.org/EllipticCurve/Q/32/a/4}{\texttt{32.a4}} and \href{https://www.lmfdb.org/EllipticCurve/Q/24/a/5}{\texttt{24.a5}}]{lmfdb}).
In addition, the action of $\Gal(\ol\Q/\Q)$ on $E_0[2]$ (by which we mean $E_0[2](\ol\Q)$) visibly factors through $\Gal(\Q(\sqrt{-1})/\Q)$, the nontrivial element of this group acting via the matrix
\[
\begin{pmatrix}
1&1\\0&1
\end{pmatrix}
\]
(with respect to a suitable basis); the same is true of $F_0[2]$, except that the action factors through $\Gal(\Q(\sqrt{-3})/\Q)$.
Therefore, the action of $\Gal(\ol\Q/\Q)$ on $\Hom(E_0[2],F_0[2])$ factors through $\Gal(\Q(\sqrt{-1},\sqrt{-3})/\Q)\cong\Gal(\Q(\sqrt{-1})/\Q)\times\Gal(\Q(\sqrt{-3})/\Q)$, and it is easily seen that there is just one nonzero invariant element (namely, the map with kernel $\{0,P\}$ and image $\{0,Q\}$, where $P$ and $Q$ are the nonzero elements of $E_0[2](\Q)$ and $F_0[2](\Q)$, respectively).

Let $t\in T(\Q)$ be the point determined by $E_0$ and $F_0$, and let $\ol t\in T(\ol\Q)$ be a geometric point above $t$.
We get a homomorphism $\Gal(\ol\Q/\Q)\cong\pi_1(\Spec(\Q),\Spec(\ol\Q))\to\pi_1(T,\ol t)$ under which the monodromy action of $\pi_1(T,\ol t)$ on $\mathscr H_{\ol t}\cong H^1(E_{0,\ol\Q},\Z/2)\otimes H^1(F_{0,\ol\Q},\mu_2)$ pulls back to the tensor product of the usual Galois actions on \'etale cohomology.
We saw in the previous paragraph that this action has one nonzero invariant, hence $\mathscr H$ has at most one nonzero global section, as desired.
\end{proof}

\begin{cor}
\label{cor:enriques-brauer-pullback}
In the setting of \tn(\ref{thing:lieberman-setup}\tn), the maps $\Br(S)\to\Br(X)$ and $\Br(S)\to\Br(E\times F)$ vanish if and only if there exists an isogeny $\phi\cln E\to F$ with $\phi(P)=0$ and $Q\in\phi(E[2])$.
\end{cor}

\begin{proof}
By \cite[Proposition 1.3]{sz-brauer}, $\Br(S)\to\Br(X)$ vanishes if and only if $\Br(S)\to\Br(E\times F)$ vanishes.
Since $\Br(S)\cong\Z/2$, the map $\Br(S)\to\Br(E\times F)$ vanishes if and only if, in the notation of Proposition \ref{prop:enriques-brauer-pullback}, the image of $\beta$ in $H^2(E\times F,\mu_2)$ is algebraic.
Because of the identification $H^1(E,\Z/2)\otimes H^1(F,\mu_2)\cong\Hom(E[2],F[2])$ of \ref{thing:ec-coh}\ref{item:ecISO}, the image of $\beta$ is algebraic if and only if such an isogeny $\phi$ exists.
\end{proof}

\section{Nonvanishing of some classes in Galois cohomology}
\label{sec:nonvanishing}

\noindent
The goal of this section is to give criteria (Corollary \ref{cor:galois-nonzero-application}) for certain classes in $H^3(S\times E,\mu_2)$ to be nonvanishing in Galois cohomology, where $S$ is an Enriques surface defined by a Lieberman involution and $E$ an elliptic curve.
The first such criterion was given by Diaz in \cite[\S4]{diaz}; we find two more criteria by using Proposition \ref{prop:enriques-brauer-pullback} above.
The key input, Proposition \ref{prop:galois-nonzero}, is modeled after \cite[Theorem 3.3]{auel-suresh}, which is in turn a modified version of an argument of Gabber (\cite[Appendice]{ct-g}).

\begin{thing}
Let $(K,v)$ be a discretely valued field with residue field $\kappa$, let $n$ be a positive integer not divisible by $\charac(\kappa)$, and let $A$ be an unramified $n$-torsion finite $\Gal_K$-module.
Then there exists a ``residue map'' in Galois cohomology
$
r_v
\cln
H^i(K,A\otimes\mu_n)
\to
H^{i-1}(\kappa,A)$.
It can be written down explicitly at the level of cochains in group cohomology (\cite[II, Appendix, \S1]{serre}) or defined by the Gysin sequence in \'etale cohomology (\cite[\S3.3]{ct-birational}).
The following can be checked at the level of cochains in group cohomology.
\begin{enumerate}
\item
The residue map is natural in $(K,v)$ in the sense that if $K'/K$ is a field extension and $v'$ is a discrete valuation on $K'$ such that $v'|_K=v$, then the square
\[
\begin{tikzcd}
H^i(K,A\otimes\mu_n)
\arrow[r,"r_v"]
\arrow[d]&
H^{i-1}(\kappa,A)
\arrow[d]
\\
H^i(K',A\otimes\mu_n)
\arrow[r,"r_{v'}"]&
H^{i-1}(\kappa',A)
\end{tikzcd}
\]
commutes, where $\kappa'$ denotes the residue field of $v'$.
\item
Let $\alpha\in H^1(K,\mu_n)$ and $\xi\in H^i(K,A)$.
If the restriction of $\xi$ to the decomposition subgroup of $v$ is pulled back from some $\tilde\xi\in H^i(\kappa,A)$, then $r_v(\alpha\smallsmile\xi)=r_v(\alpha)\cdot\tilde\xi$.
\end{enumerate}
\end{thing}

\noindent
In the following, given a discrete valuation ring $\mf o$, an integral $\mf o$-scheme $\mc X$ with integral special fiber, and a map $\mf o\to L$ for some field $L$, we write $L(\mc X)$ for the function field of the base-change $\mc X_L$.

\begin{prop}
\label{prop:galois-nonzero}
Let $X,Y_1,\dots,Y_b$ be integral $k$-schemes, and let $\xi\in H^a(X,\mu_n^{\otimes a'})$ and $\alpha_i\in H^1(Y_i,\mu_n)$ for each $i$.
Assume that there exist
\begin{itemize}
\item
a discrete valuation subring $\mf o$ of $k$ with residue field $\kappa$ such that $\charac(\kappa)\nmid n$,
\item
integral $\mf o$-schemes $\mc X,\mc Y_1,\dots,\mc Y_b$ which recover $X,Y_1,\dots,Y_b$ upon base-change to $k$ and have integral special fibers $\mc X_\kappa,\mc Y_{1,\kappa},\dots,\mc Y_{b,\kappa}$, and
\item
classes $\xi_{\mf o}\in H^a(\mc X,\mu_n^{\otimes a'})$ and $\alpha_{i,\mf o}\in H^1(\mc Y_i,\mu_n)$ \end{itemize}
such that
\begin{conditions}
\item
$\xi_{\mf o},\alpha_{1,\mf o},\dots,\alpha_{b,\mf o}$ recover $\xi,\alpha_1,\dots,\alpha_b$ upon pul{}lback to $X,Y_1,\dots,Y_b$, respectively.
\item
\label{item:gnX}
$\xi_{\mf o}$ does not vanish in $H^a(\kappa'(\mc X),\mu_n^{\otimes a'})$ for any field extension $\kappa'/\kappa$.
\item
\label{item:gnA}
for each $i$, the image of $\alpha_{i,\mf o}$ under
\[
H^1(\mc Y_i,\mu_n)
\to
H^1(\kappa(\mc Y_i),\mu_n)
\xrightarrow{r_{v_i}}
\Z/n
\]
is invertible for some valuation $v_i$ on $\kappa(\mc Y_i)$ induced by a prime divisor of $\mc Y_{i,\kappa}$.
\end{conditions}
Then $\xi\boxtimes\alpha_1\boxtimes\cdots\boxtimes\alpha_b$ does not vanish in $H^{a+b}(k(X\times\prod_iY_i),\mu_n^{\otimes a'+b})$.
\end{prop}

\begin{proof}
Let $k_0\subseteq k$ denote the fraction field of $\mf o$.
We first reduce to the case when $k$ is algebraic over $k_0$.
Let $\mf B$ be a transcendence basis of $k/k_0$.
Then $v$ extends to a discrete valuation on $k_0(\mf B)$ by putting
\[
v{\left(
\sum_za_z\cdot z
\right)}
\ceq
\min_{z}v(a_z),
\]
where $z$ ranges over all monomials in the elements of $\mf B$ and each $a_z$ lives in $k_0$.
All of the hypotheses remain true if we replace $k_0$ by $k_0(\mf B)$.

Put $\mc Z\ceq\mc X\times\prod_i\mc Y_i$ (product over $\mf o$) and $\theta_{\mf o}\ceq\xi_{\mf o}\boxtimes\alpha_{1,\mf o}\boxtimes\cdots\boxtimes\alpha_{b,\mf o}\in H^{a+b}(\mc Z,\mu_n^{\otimes a'+b})$.
We will show that the image of $\theta_{\mf o}$ in $H^{a+b}(k_0(\mc Z),\mu_n^{\otimes a'+b})$ is nonzero.
This suffices to prove the present proposition, because the same argument will show that the image of $\theta_{\mf o}$ remains nonzero when $k_0$ is replaced by any finite extension, hence also (thanks to the previous paragraph) by $k$ itself.

The prime divisor $\mc Z_\kappa$ of $\mc Z$ gives rise to an extension of $v$ to a discrete valuation on $k_0(\mc Z)$ with residue field $\kappa(\mc Z)$.
Let $\mf o(\mc Z)$ denote its valuation ring, and let $\mf o(\mc Z)_v$ and $k_0(\mc Z)_v$ be the respective completions at $v$.
We get a commutative diagram
\[
\begin{tikzcd}
&[-20pt]
H^{a+b}(\mc Z_{k_0},\mu_n^{\otimes a'+b})\arrow[r]&
H^{a+b}(k_0(\mc Z),\mu_n^{\otimes a'+b})\arrow[r]&
H^{a+b}(k_0(\mc Z)_v,\mu_n^{\otimes a'+b})
\\
\theta_{\mf o}\arrow[r,phantom,"\in"]&
H^{a+b}(\mc Z,\mu_n^{\otimes a'+b})\arrow[u]\arrow[r]\arrow[d]&
H^{a+b}(\mf o(\mc Z),\mu_n^{\otimes a'+b})\arrow[u]\arrow[r]\arrow[d]&
H^{a+b}(\mf o(\mc Z)_v,\mu_n^{\otimes a'+b})\arrow[u,hook,"(*)"']\arrow[d,"\sim"]
\\&
H^{a+b}(\mc Z_\kappa,\mu_n^{\otimes a'+b})\arrow[r]&
H^{a+b}(\kappa(\mc Z),\mu_n^{\otimes a'+b})\arrow[r,equals]&
H^{a+b}(\kappa(\mc Z),\mu_n^{\otimes a'+b}),
\end{tikzcd}
\]
where the arrow $(*)$ is injective by \cite[Appendice, Lemme A5]{ct-g}.\footnote{
For details on the splitting in the proof of \tit{loc.\ cit.}, see MathOverflow questions \href{https://mathoverflow.net/q/386147}{\texttt{386147}} and \href{https://mathoverflow.net/q/462675}{\texttt{462675}}.
}
In view of the diagram, it suffices to show that the image $\theta'$ of $\theta_{\mf o}$ in $H^{a+b}(\kappa(\mc Z),\mu_n^{\otimes a'+b})$ is nontrivial.
Let $\xi'$ and $\alpha_i'$ be, respectively, the images of $\xi_{\mf o}$ and $\alpha_{i,\mf o}$ in $H^a(\kappa(\mc X),\mu_n^{\otimes a'})$ and $H^1(\kappa(\mc Y_i),\mu_n)$, so that $\theta'=\xi'\boxtimes\alpha_1'\boxtimes\cdots\boxtimes\alpha_b'$.
Let $W_i$ be a prime divisor of $\mc Y_{i,\kappa}$ inducing the valuation $v_i$ of item \ref{item:gnA} of the statement.
Then the prime divisor $\mc X_\kappa\times\mc Y_{1,\kappa}\times\cdots\times\mc Y_{b-1,\kappa}\times W_b$ of $\mc Z_\kappa$ induces an extension of $v_b$ to a valuation $\tilde v_b$ on $\kappa(\mc Z)$, and
\[
r_{\tilde v_b}(\theta')
=
r_{v_b}(\alpha_b')\cdot f^*(\xi'\boxtimes\alpha_1'\boxtimes\cdots\boxtimes\alpha_{b-1}');
\]
here $r_{v_b}(\alpha_b')\in\Z/n$ is invertible by \ref{item:gnA}, and $f^*$ denotes pullback in Galois cohomology from $\kappa(\mc X\times\mc Y_1\times\cdots\times\mc Y_{b-1})$ to $\kappa(\mc X\times\mc Y_1\times\cdots\times\mc Y_{b-1}\times W_b)$.
Repeatedly taking residues in a similar way, we find that $(r_{\tilde v_1}\circ\cdots\circ r_{\tilde v_b})(\theta')$ is, up to an element of $(\Z/n)^\times$, the pullback of $\xi'$ from $\kappa(\mc X)$ to $\kappa(\mc X\times W_1\times\cdots\times W_b)=\kappa(W_1\times\cdots\times W_b)(\mc X)$, which is nonzero by \ref{item:gnX}.
\end{proof}

\begin{cor}
\label{cor:galois-nonzero-application-0}
Let $k$ be an algebraical{}ly closed field of characteristic not $2$.
Let $E$ be an el{}liptic curve over $k$, $S$ an Enriques surface over $k$, $\beta\in H^2(S,\mu_2)$ any non-algebraic class, and $P\in E[2]\setminus\{O\}$.
Assume that there exists a subfield $k_0$ of $k$ and a discrete valuation $v$ on $k_0$ with residue characteristic not $2$ such that $(E,P)$ and $S$ descend to $k_0$ and have, respectively, multiplicative and good reduction mod $v$, and further that $P$ reduces to a smooth point.
Then the image of $\beta\otimes\alpha_P$ in $H^3(k(S\times E),\mu_2^{\otimes2})$ is nonzero, where $\alpha_P$ is as \ref{thing:ec-coh}\ref{item:ecISO}.
\end{cor}

\begin{proof}
This is a slight generalization of the situation of \cite[\S4]{diaz}, and our argument resembles that of Diaz.
We will show that the assumptions of Proposition \ref{prop:galois-nonzero} hold with $b=1$ and $(X,Y_1,\xi,\alpha_1)=(S,E,\beta,\alpha_P)$.

We first check condition \ref{item:gnA}.
After finitely enlarging $k_0$, we will assume that $(E,P)=(E_\nu,P_\nu)$ for some $\nu\in k\setminus\{0,1\}$.
Since $E_\nu$ has bad reduction, $v(\nu)\neq0$, and we will assume that $v(\nu)>0$ after replacing $\nu$ with $\nu^{-1}$ if necessary.
As noted in \eqref{thing:star} and \ref{thing:ec-coh}\ref{item:ec2}, the class $\alpha_{P_{\nu}^\vee}$ corresponds to the $\Z/2$-torsor $E_{\nu^*}\to E_{\nu^*}/P_{\nu^*}\cong E_\nu$.
Let $\mc E_{\nu^*}$ denote the smooth locus of the closed subscheme of $\mbb P^2_{\mf o}$ cut out by the Legendre equation $y^2=x(x-1)(x-\nu^*)$.
Since $v(\nu^*)=0$, the point $P_{\nu^*}$ spreads out to $\mc P_{\nu^*}\in\mc E_{\nu^*}(\mf o)$.
The group structure of $E_{\nu^*}$ extends to $\mc E_{\nu^*}$ (\cite[Ch.\ IV, Theorem 5.3]{silverman-II}), so we get another $\Z/2$-torsor $\mc E_{\nu^*}\to\mc E_{\nu^*}/\mc P_{\nu^*}$, hence a class $\alpha_{P_\nu^\vee,\mf o}\in H^1(\mc E_{\nu^*}/\mc P_{\nu^*},\Z/2)$ whose image in $H^1(E_\nu,\Z/2)$ is $\alpha_{P_\nu^\vee}$.
Since $v(\nu^*-1)>0$, the special fiber of $\mc E_{\nu^*}\to\mc E_{\nu^*}/\mc P_{\nu^*}$ is isomorphic to the double cover $\mbb G_\mr{m}\to\mbb G_\mr{m}$.
Therefore, $\alpha_{P_\nu,\mf o}$, the image of $\alpha_{P_\nu^\vee,\mf o}$ under $\Z/2\iso\mu_2$, satisfies condition \ref{item:gnA} of Proposition \ref{prop:galois-nonzero} (for which we take $\mc Y_1=\mc E_{\nu^*}/\mc P_{\nu^*}$) by the exact sequence
\[
H^1(\kappa,\mu_2)
\to
H^1(\kappa(t),\mu_2)
\xrightarrow{\prod_xr_{v_x}}
\prod_{x}\Z/2,
\]
where $x$ ranges over the closed points of $\mbb P^1_\kappa$ and $v_x$ is the associated discrete valuation (\cite[II, Appendix, \S4]{serre}).

Let $\mc S$ denote the smooth $\mf o$-model of $S$; we will take $\mc X=\mc S$.
We claim that, after finitely extending $\mf o$, there exists $\beta_{\mf o}\in H^2(\mc S,\mu_2)$ whose image in $H^2(S,\mu_2)$ is $\beta$ and whose image in $H^2(\mc S_{\ol\kappa},\mu_2)$ is not algebraic.
Granted this, the image of $\beta_{\mf o}$ is nonzero in $H^2(\kappa'(\mc S),\mu_2)$ for any field extension $\kappa'/\kappa$, because $\Br(\mc S_{\ol\kappa})\to\Br(\mc S_{\kappa'})$ is an isomorphism when $\kappa'$ is algebraically closed, so $\beta_{\mf o}$ satisfies condition \ref{item:gnX} of Proposition \ref{prop:galois-nonzero}.
To find such a $\beta_{\mf o}$, let $\ol{\mf o}$ be a strict Henselization of $\mf o$ inside $k$ (cf.\ \cite[Ch.\ I, Example 4.10 and p.\ 38]{milne-book}), and let $f\cln\mc S_{\ol{\mf o}}\to\Spec(\ol{\mf o})$ be the obvious map.
By the theorems on smooth and proper base change, the maps $\eta\cln H^2(\mc S_{\ol{\mf o}},\mu_{2^n})\to H^2(S,\mu_{2^n})$ are isomorphisms (cf.\ the proof of \cite[Lemma 7.8.2]{fu}).
Now $H^2(\mc S_{\ol{\mf o}},\mu_2)$ is the filtered colimit of $H^2(\mc S_{\mf o'},\mu_2)$ for certain finite ring maps $\mf o\to\mf o'$ (\cite[Ch.\ III, Lemma 1.16]{milne-book}); we choose $\beta_{\mf o}$ by taking a preimage of $\beta$ along $\eta$ and descending to some $\mf o'$ (and then replace $\mf o$ with $\mf o'$ and $k_0$ with the fraction field of $\mf o'$).
To show that the image of $\beta_{\mf o}$ in $H^2(\mc S_{\ol\kappa},\mu_2)$ is not algebraic, observe that the restriction maps $H^2(\mc S_{\ol{\mf o}},\mu_{2^n})\to H^2(\mc S_{\ol\kappa},\mu_{2^n})$ are also isomorphisms by the proper base change theorem (\cite[Corollary 7.3.3]{fu}).
Thus, since being algebraic is the same as lifting to a class with $\mu_4$-coefficients (see \eqref{thing:enriques-brauer}), and $\beta$ is not algebraic, the same is true for the image of $\beta_{\mf o}$.
\end{proof}

\noindent
For the notation and terminology in the following statement, see \eqref{thing:legendre}, \eqref{defn:suitable}, and \eqref{thing:ec-coh}.

\begin{cor}
\label{cor:galois-nonzero-application}
Let $k$ be an algebraical{}ly closed field of characteristic not $2$.
Let $\lambda,\mu,\nu\in k\setminus\{0,1\}$, and let $\beta\in H^2(S_{\lambda,\mu},\mu_2)$ be any non-algebraic class.
Assume that $(\lambda,\mu,\nu)$ is suitable.
Then the image of $\beta\otimes\alpha_{P_\nu}$ in $H^3(k(S_{\lambda,\mu}\times E_\nu),\mu_2^{\otimes2})$ is nonzero.
\end{cor}

\begin{proof}
Case \ref{thing:legendre}\ref{item:gnaGGB} is a special case of Corollary \ref{cor:galois-nonzero-application-0}.
In general, consider the pullback of $\beta\otimes\alpha_{P_\nu}$ to $E_\lambda\times E_\mu\times E_\nu$.
By Proposition \ref{prop:enriques-brauer-pullback}, it equals $\alpha_{P_\lambda^\vee}\otimes\alpha_{P_\mu}\otimes\alpha_{P_\nu}$.
We conclude the proof by using the same arguments as given above to apply Proposition \ref{prop:galois-nonzero}.
Specifically, in case \ref{defn:suitable}\ref{item:gnaBB}, the assignment $(X,Y_1,Y_2,\xi,\alpha_1,\alpha_2)=(E_{\lambda_1},E_{\lambda_2},E_{\lambda_3},\alpha_{P_{\lambda_1}^\vee},\alpha_{P_{\lambda_2}},\alpha_{P_{\lambda_3}})$ works, and in case \ref{defn:suitable}\ref{item:gnaB}, $(X,Y_1,\xi,\alpha_1)=(E_{\lambda_1}\times E_{\lambda_2},E_{\lambda_3},\alpha_{P_{\lambda_1}^\vee}\otimes\alpha_{P_{\lambda_2}},\alpha_{P_{\lambda_3}})$ works.
Actually, the argument simplifies, because in these cases $\xi_{\mf o}$ can be written down explicitly.
The clunky hypothesis ``moreover there is no isogeny\dots'' of \ref{defn:suitable}\ref{item:gnaB} forces hypothesis \ref{item:gnA} of Proposition \ref{prop:galois-nonzero} to hold; more precisely, it guarantees that the class \smash{$\alpha_{P_{\lambda_1}^\vee}\otimes\alpha_{P_{\lambda_2}}$} spreads out to $\xi_{\mf o}$ whose image in $H^2(E_{\ol{\lambda_1}}\times E_{\ol{\lambda_2}},\mu_2)$ is not algebraic, hence nonzero in the Brauer group after any algebraically closed extension of the residue field (as in the proof of Corollary \ref{cor:galois-nonzero-application-0}).
\end{proof}

\begin{rmk}
In some cases one can use an additional application of Proposition \ref{prop:galois-nonzero} to show that the isogeny condition of \ref{defn:suitable}\ref{item:gnaB} holds, for example if there is a valuation on a subfield of the \tit{residue field} of $v$ with respect to which both $\ol{\lambda_1}$ and $\ol{\lambda_2}$ have positive valuation.
\end{rmk}

\section{Torsion motives}
\label{sec:torsion-motives}

\noindent
Continue to assume that $k$ is algebraically closed, and let $p$ denote the exponential characteristic of $k$.
For the definition of a Chow motive, see e.g.\ \cite[Ch.\ 3--4]{andre}.
For a smooth projective variety $X$ and a ring $\Lambda$, we write $\fh(X)_\Lambda$ for the Chow motive of $X$ with coefficients in $\Lambda$, or simply $\fh(X)$ if $\Lambda=\Z$.
We denote by $\mbf1\ceq\fh(\Spec(k))_\Lambda$ the unit object.
We will use the usual ``contravariant'' normalization of the category of Chow motives, i.e.\ the functor from smooth projective varieties to Chow motives to be contravariant and $\fh(\mbb P^1)\cong\mbf1\oplus\mbf1(-1)$.

\begin{defn}
A Chow motive $M$ with coefficients in $\Z[1/a]$ is \tit{torsion} if $(\End(M),+)$ is a torsion group.
If $t\cdot\End(M)=0$ for some positive integer $t$, we say that $M$ is $t$-torsion, or simply that $t\cdot M=0$.
We call the smallest such $t$ the \tit{torsion order} of $M$.
A torsion motive $M$ always has a finite torsion order, namely, the order of $\id_M$ in $(\End(M),+)$.
\end{defn}

\noindent
In the remainder of this section, we explain that two well-known facts hold in the setting of torsion motives.
The first concerns torsion in $\CH^2$, following Merkurjev--Suslin (\cite[\S18]{ms}), and the second concerns the relationship between unramified cohomology and algebraicity of cohomology classes, following Colliot-Th\'el\`ene--Voisin (\cite{ct-v}).
We remark that, for the purposes of this paper, we could make do with citing statements already in the literature.
However, the ``motivic'' Propositions \ref{prop:bloch-kato} and \ref{prop:ur} will be useful in future work.

\begin{prop}
\label{prop:bloch-kato}
Let $X$ be a smooth projective variety over $k$, let $M$ be a torsion summand of $\fh(X)_{\Z[1/a]}$, and let $\ell\neq p$ be a prime.
Then the cycle-class map $\cl_{M,\ell}^i\cln\CH^i(M)[\ell^\oo]\to H^{2i}(M,\Z_\ell(i))$ is bijective for $i=1$ and injective for $i=2$.
\end{prop}

\begin{proof}
By assumption, $\CH^1(M)$ is a torsion summand of $\Pic(X)$ of finite torsion order.
Since $\Pic^\circ(X)=\Ker(\CH^1(X)\to \NS(X))$ is the group of points of an Abelian variety over an algebraically closed field, it has no direct summand of finite torsion order, so the map $\CH^1(M)\to\NS(X)$ is injective.
Thus the map $\cl_{M,\ell}^1\cln\CH^1(M)[\ell^\oo]\to H^2(M,\Z_\ell(1))$ is injective as well.
Now the Kummer sequence shows that $\Coker(\cl_{M,\ell}^1)$ embeds into $T_\ell(\Br(X))$, the $\ell$-adic Tate module of the Brauer group of $X$.
But $T_\ell(\Br(X))$ is torsion-free, so we must have $\Coker(\cl_{M,\ell}^1)=0$ since $M$ is torsion.

For the case $i=2$, we will use Voevodsky's Zariski (or, equivalently, Nisnevich) motivic cohomology $H_\mot^{i,j}$.
There is a functorial map $H_\mot^{i,j}(X,\Z/\ell^n)\to H^i(X,\mu_{\ell^n}^{\otimes j})$ which, by the Beilinson--Lichtenbaum conjecture (now a theorem of Voevodsky), is injective if $i=j+1$ (and an isomorphism if $i\leq j$); see \cite[Corollary 1.31]{hw}.
Letting $\Z_{\gen\ell}$ denote the localization of $\Z$ at the prime ideal $\ell\Z$, there is a map $H_\mot^{i,j}(X,\Q/\Z_{\gen\ell})\to H^i(X,\Q_\ell/\Z_\ell(j))$ because Zariski cohomology commutes with direct limits.
Also, there is a map $f\cln H_\mot^{i,j}(X,\Z_{\gen\ell})\to H^i(X,\Z_\ell(j))$ which is the inverse limit of compositions of the form $H_\mot^{i,j}(X,\Z_{\gen\ell})\to H_\mot^{i,j}(X,\Z/\ell^n)\to H^i(X,\mu_{\ell^n}^{\otimes j})$.
We get a commutative diagram
\begin{equation}
\label{eq:m2e}
\begin{tikzcd}
&
H_\mot^{3,2}(X,\Q/\Z_{\gen\ell})\arrow[d,hook]\arrow[r]&
H_\mot^{4,2}(X,\Z_{\gen\ell})_\tors\arrow[r]\arrow[d,"f"]&0\\
H^3(X,\Q_\ell(2))\arrow[r]&
H^3(X,\Q_\ell/\Z_\ell(2))\arrow[r]&
H^4(X,\Z_\ell(2))
\end{tikzcd}
\end{equation}
with exact rows coming from the short exact sequences $0\to\Z_{\gen\ell}\to\Q\to\Q/\Z_{\gen\ell}\to0$ and $0\to\Z_\ell\to\Q_\ell\to\Q_\ell/\Z_\ell\to0$, respectively; the downward arrow on the left is injective by the Beilinson--Lichtenbaum conjecture.
Moreover, $f$ is the cycle-class map under the identification $H_\mr{mot}^{2i,i}(X,\Z)\cong\CH^i(X)$ of \cite[Corollary 2]{voevodsky}.

The functor from smooth projective varieties to Voevodsky's triangulated category of motives factors through the category of Chow motives (\cite[Proposition 2.1.4]{voevodsky}).
In particular, we obtain the diagram \eqref{eq:m2e} with $X$ replaced by $M$ everywhere.
The claim then follows from the fact that $H^3(M,\Q_\ell(2))=0$ because $M$ is torsion.
\end{proof}

\begin{prop}
\label{prop:ur}
Let $M$ be a torsion summand of $\fh(X)_{\Z[1/a]}$, where $X$ is a connected smooth projective variety over $k$.
Then, for any prime $\ell\neq p$ and any $n\geq1$, the boundary map $\partial\cln H^3(M,\mu_{\ell^n}^{\otimes2})\to H^4(M,\Z_\ell(2))$ induces an injective map
\begin{equation}
\label{eq:ur}
\Img\big(
H^3(M,\mu_{\ell^n}^{\otimes2})
\to
H^3(k(X),\mu_{\ell^n}^{\otimes2})
\big)
\into
\Coker(\mr{cl}_{M,\ell}^2),
\end{equation}
where $\mr{cl}_{M,\ell}^2$ denotes the cycle-class map $\CH^2(M)\to H^4(M,\Z_\ell(2))$.
\end{prop}

\begin{proof}
The proof relies on the theory of unramified cohomology, which, according to \cite{schreieder-refined}, may be defined as follows (though the specifics of the definition are not important here).
For $A$ being either $\mu_{\ell^n}^{\otimes j}$ or $\Z_\ell(j)$, we define $H^i(F_cX,A)$ to be the colimit of $H^i(U,A)$ as $U$ ranges over open subsets of $X$ of codimension $>c$, and then define $H_\ur^i(X,A)$ to be the image of $H^i(F_1X,A)\to H^i(F_0X,A)$.\footnote{
In general, Schreieder defines unramified cohomology in terms of a Borel--Moore homology theory; since $X$ is smooth, this agrees with the definition above by Poincar\'e duality (\cite[\S1.4]{schreieder-refined}).
An argument with the Gysin sequence shows that it also agrees with the usual definitions of unramified cohomology (\cite[equation (1.4)]{schreieder-refined}).
}
We have maps $H^i(X,A) \to H_\ur^i(X,A)$ and, for $X$ integral, an injection $H_\ur^i(X,\mu_{\ell^n}^{\otimes j})\into H^i(k(X),\mu_{\ell^n}^{\otimes j})$ (indeed, the natural map $H^i(F_0X,A)\to H^i(k(X),\mu_{\ell^n}^{\otimes j})$ is an isomorphism).

The link between unramified cohomology and non-algebraicity of cohomology classes is the existence of a commutative diagram with exact rows of the form
\begin{equation}
\label{eq:ur-diagram}
\begin{tikzcd}
&
H^3(X,\mu_{\ell^n}^{\otimes2})
\arrow[d]\arrow[r,"\partial"]
&
H^4(X,\Z_\ell(2))
\arrow[d]
\\
H_\ur^3(X,\Z_\ell(2))
\arrow[r]
&
H_\ur^3(X,\mu_{\ell^n}^{\otimes2})
\arrow[r]&
\Coker(\cl_{X,\ell}^2)
\arrow[r]&
0,
\end{tikzcd}
\end{equation}
where $\partial$ is the usual boundary map in $\ell$-adic cohomology.
This was first proven by Colliot-Th\'el\`ene and Voisin (for $k=\C$ and Betti cohomology; see \cite[Th\'eor\`eme 3.7]{ct-v}), and an ``elementary'' argument is explained in \cite[\S3]{schreieder-refined}.
Note that Schreieder works with the exact sequence $0\to\Z\to\Q\to\Q/\Z\to0$ in Betti cohomology, but the same argument works for the exact sequence $0\to\Z_\ell(2)\to\Z_\ell(2)\to\mu_{\ell^n}^{\otimes2}\to0$ in $\ell$-adic cohomology.
In particular, the map \eqref{eq:ur} exists.

A consequence of Bloch--Ogus's proof of the Gersten conjecture for $\ell$-adic cohomology (\cite{bloch-ogus}) is that the assignment $Y\mapsto H_\ur^i(Y,A)$ for smooth projective $Y$ extends to a functor on the category of Chow motives; see \cite[Corollary 1.7]{schreieder-moving} for a concrete (and more general) statement.
Therefore, we obtain the diagram \eqref{eq:ur-diagram} with $X$ replaced by $M$ everywhere.
By \cite[Lemme 3.12]{kahn-classes} (see also \cite[Th\'eor\`eme 3.1]{ct-v}), $H_\ur^3(M,\Z_\ell(2))$ is torsion-free, hence trivial (because $M$ is torsion).
This proves the injectivity of \eqref{eq:ur}.
\end{proof}

\section{Torsion in the motive of a surface}
\label{sec:surface-motives}

\noindent
Continue to assume that $k$ is algebraically closed of exponential characteristic $p$.
In this section, we draw from Proposition \ref{prop:bloch-kato} numerous consequences for torsion summands of the motives of surfaces.
Then we explain the only known source of such torsion motives, namely surfaces admitting a decomposition of the diagonal (Proposition \ref{prop:dod}).\footnote{
In \cite{vishik-torsionspaces}, Vishik constructs many other torsion objects in the stable motivic homotopy category, but these do not give rise to Chow motives, as noted in \cite[Remark 2.2]{vishik-torsionspaces}.
}

If $M$ is the summand of $\fh(X)$ cut out by a correspondence $\epsilon$, let us denote by $M^\mr{t}$ the summand cut out by the transposed correspondence $\epsilon^\mr{t}$ (not to be confused with $M^\vee\ceq M^\mr{t}({\dim(X)})$, the usual duality functor on Chow motives.)

\begin{thing}
\label{thing:torsion-in-surface}
Let $S$ be a smooth projective surface over $k$, and let $M$ be a torsion summand of $\fh(S)_{\Z[1/a]}$.
We will assume that the torsion order of $M$ is a power of a prime $\ell$ different from $p$.\footnote{
Torsion motives have canonical ``primary decompositions'', so we can take the ``$\ell$-part'' of any torsion summand of $\fh(S)_{\Z[1/a]}$.
}
By Roitman's theorem (\cite{roitman}), $\CH_0(M)=0$.
Thus $\CH^i(M)=0$ for $i\neq 1$, and the map $\CH^1(M)\to H^2(M,\Z_\ell)$ is an isomorphism by Proposition \ref{prop:bloch-kato}.
Also, $H^i(M,\Z_\ell)=0$ for $i\neq2,3$.
Poincar\'e duality induces isomorphisms $H^2(M,\Z_\ell)^\vee\cong H^3(M^\mr{t},\Z_\ell)$ and $H^3(M,\Z_\ell)^\vee\cong H^2(M^{\mr t},\Z_\ell)$, where $(-)^\vee\ceq\Hom(-,\Q/\Z)$; see \eqref{thing:cohfin}--\eqref{thing:motive-pairing} below.
\end{thing}

\begin{thing}
\label{thing:tis-hom}
Continuing in the setting of \eqref{thing:torsion-in-surface}, let $S'$ be another smooth projective surface over $k$, and let $M'$ be a torsion summand of $\fh(S')_{\Z[1/a]}$, and assume that the torsion order of $M'$ is also a power of $\ell$.
Then Proposition \ref{prop:bloch-kato} and the K\"unneth theorem for $\ell$-adic cohomology (see Proposition \ref{prop:good-kunneth} below) yield the following commutative diagram with exact rows, where we have abbreviated $H^i(-,\Z_\ell)$ by $H^i(-)$:
\[
\begin{tikzcd}
[column sep=small]
0\arrow[r]&
\CH^1(M^{\mr t})\otimes\CH^1(M')\arrow[r]\arrow[d,"\sim"']&
\CH^2(M^{\mr t}\otimes M')\arrow[r]\arrow[d,hook]&
\tn{Coker}\arrow[r]\arrow[d,hook]&
0\\
0\arrow[r]&
H^2(M^{\mr t})\otimes H^2(M')\arrow[r]&
H^4(M^{\mr t}\otimes M')\arrow[r]&
\displaystyle\bigoplus_{i=2,3}\Hom(H^i(M),H^i(M'))\arrow[r]&
0.
\end{tikzcd}
\]
Note that $\CH^2(M^{\mr t}\otimes M')=\Hom(M,M')$.

Moreover, the image of $\Hom(M,M')$ in $\bigoplus_i\Hom(H^i(M),H^i(M'))$ is dual to the image of $\Hom(M'^{\mr t},M^{\mr t})$ in $\bigoplus_i\Hom(H^i(M'^{\mr t}),H^i(M^{\mr t}))$ under Poincar\'e duality.
More precisely, if $\phi\cln M\to M'$ acts as $(f_2,f_3)$ on cohomology, then the transpose $\phi^{\mr t}$ defines a morphism $M'^{\mr t}\to M^{\mr t}$, which acts as $(f_3^\vee,f_2^\vee)$ on cohomology.
\end{thing}

\begin{lem}
\label{lem:action-on-coh}
Let $M$ and $M'$ be as in \tn(\ref{thing:tis-hom}\tn).
\begin{enumerate}
\item\label{item:aocTO}
The torsion order $M$ equals that of $H^*(M,\Z_\ell)$.
\item\label{item:aocAIA}
Suppose $\phi\cln M\to M$ induces an automorphism of $H^*(M,\Z_\ell)$.
Then $\phi$ is an automorphism.
\item\label{item:aocIII}
Suppose $\phi\cln M\to M'$ induces an isomorphism $H^*(M,\Z_\ell)\to H^*(M',\Z_\ell)$.
Assume there exists some morphism $M'\to M$ which induces an isomorphism $H^*(M',\Z_\ell)\to H^*(M,\Z_\ell)$.
Then $\phi$ is an isomorphism.
\item\label{item:aocIIIS}
Assume that $M^{\mr t}\cong M$ and $M'^{\mr t}\cong M'$, and suppose $\phi\cln M\to M'$ induces an isomorphism $H^*(M,\Z_\ell)\to H^*(M',\Z_\ell)$.
Then $\phi$ is an isomorphism.
\item\label{item:aocOI}
Suppose $\epsilon_1,\dots,\epsilon_n\cln M\to M$ induce orthogonal idempotents in $\End(H^*(M,\Z_\ell))$.
Then there exist $\epsilon'_1,\dots,\epsilon'_n\cln M\to M$ which are themselves orthogonal idempotents in $\End(M)$ such that $\epsilon'_i$ and $\epsilon_i$ induce the same endomorphism of $H^*(M,\Z_\ell)$ for each $i$.
\end{enumerate}
\end{lem}

\begin{proof}
\ref{item:aocTO}.
Obviously, the torsion order of $M$ is a multiple of that of $H^*(M,\Z_\ell)$.
On the other hand, the torsion order of $M$ divides that of $H^4(M^{\mr t}\otimes M,\Z_\ell)$ by Proposition \ref{prop:bloch-kato}.
Since the K\"unneth sequence splits, the latter has torsion order equal to that of $H^*(M,\Z_\ell)$ by the isomorphism $H^2(M^{\mr t},\Z_\ell)\cong H^3(M,\Z_\ell)^\vee$ of Poincar\'e duality.
This proves \ref{item:aocTO}.

For any torsion summands $N$ and $N'$ of the motives of smooth projective surfaces, let $J(N,N')\subseteq\Hom(N,N')$ denote the image of $\CH^1(N^*)\otimes\CH^1(N')$.
Now let $N''$ be a third, and let $\phi\cln N\to N'$ and $\phi'\cln N'\to N''$.
One easily checks that $\phi'\circ\phi\in J(N,N'')$ if either $\phi\in J(N,N')$ or $\phi'\in J(N',N'')$ and that $\phi'\circ\phi=0$ if $\phi\in J(N,N')$ and $\phi'\in J(N',N'')$.

\ref{item:aocAIA}.
Since $H^*(M,\Z_\ell)$ is finite, some power $\phi^{\circ r}$ of $\phi$ acts as the identity on $H^*(M,\Z_\ell)$, i.e.\ $\phi^{\circ r}=1+\epsilon$ for some $\epsilon\in J(M,M)$. 
Since $\epsilon^{\circ2}=0$, we have $(1-\epsilon)\circ\phi^{\circ r}=\phi^{\circ r}\circ(1-\epsilon)=1$, so $\phi$ is an automorphism.

\ref{item:aocIII}.
This follows from \ref{item:aocAIA} upon considering $\phi\circ\psi$ and $\psi\circ\phi$.

\ref{item:aocIIIS}.
The given isomorphisms allow us to view $\phi^{\mr t}$ as an element of $\Hom(M',M)$, and this map acts as an isomorphism on cohomology.
So, the desired result follows from \ref{item:aocIII}.

\ref{item:aocOI}.
The kernel of $\End(M)\to\End(H^*(M,\Z_\ell))$ is $J(M,M)$, a two-sided nilpotent ideal, so this follows from \cite[Lemma 5.4]{jannsen}.
\end{proof}

\begin{rmk}
\label{rmk:aocIIISC}
We will frequently use the following special case of Lemma \ref{lem:action-on-coh}\ref{item:aocIII}:
if $S=S'$ and $H^*(M,\Z_\ell)=H^*(M',\Z_\ell)$ as summands of $H^*(S,\Z_\ell)$, then the compositions $M\to\fh(S)_{\Z[1/a]}\to M'$ and $M'\to\fh(S)_{\Z[1/a]}\to M$ are isomorphisms.
\end{rmk}

\begin{rmk}
Lemma \ref{lem:action-on-coh}\ref{item:aocIII} is false without the assumption that there exists a morphism $M'\to M$ inducing an isomorphism on cohomology, as we will show in Example \ref{eg:nsni} below. 
\end{rmk}

\noindent
The following proposition, which provides many examples of torsion motives, is essentially \cite[Proposition 3.4]{sy}, which itself is a generalization of \cite[Propositions 2.2 and 2.3]{phantom}.

\begin{prop}
\label{prop:dod}
Let $S$ be a smooth projective surface over $k$.
The fol{}lowing are equivalent:
\begin{conditions}
\item
\label{item:dDOD}
$\CH_0(S_{k'})\cong\Z$ for any algebraical{}ly closed field $k'/k$.\footnote{
See \cite[Proposition 3.1.1]{birational-I} for several other equivalent conditions of the same flavor.}
\item
\label{item:dDECOMP}
There is a decomposition $\fh(S)_{\Z[1/p]}=\mbf1\oplus\mbf1(-1)^{\rho(S)}\oplus\mbf1(-2)\oplus\ft(S)$ in which $\ft(S)$ is torsion, where $\rho(S)$ is the Picard number of $S$.
\end{conditions}
If these conditions hold, then the fol{}lowing are true about $\ft(S)$.
\begin{enumerate}
\item
\label{item:dWDS}
$\ft(S)$ is wel{}l-defined up to isomorphism and satisfies $\ft(S)^{\mr t}\cong\ft(S)$.
\item
\label{item:dINV}
We have $H^*(\ft(S),\Z_\ell)=H^*(S,\Z_\ell)_\tors$ for any prime $\ell\neq p$ and $\CH^*(\ft(S))=\CH^1(S)_\tors$.
\end{enumerate}
\end{prop}

\begin{proof}
Assume \ref{item:dDECOMP}.
Because the decomposition is stable under extension of the base field, for any algebraically closed field $k'/k$, we have $\CH_0(S_{k'})_{\Z[1/p]}\cong\Z[1/p]\oplus A_{k'}$ for some group $A_{k'}$ of finite torsion order.
This implies \ref{item:dDOD} by Roitman's theorem (\cite{roitman}) and its strengthening in positive characteristic due to Milne (\cite[Theorem 0.1]{milne-roitman}).

Now assume \ref{item:dDOD}.
Since the Albanese variety of $S$ is trivial, the Picard scheme of $S$ is zero-dimensional.
In particular, the group of numerically trivial cycles in $\CH^1(S)$ is precisely the torsion subgroup.

Pick $x\in S(k)$ and divisors $C_1,D_1,\dots,C_{\rho(S)},D_{\rho(S)}$ on $S$ for which $([C_1],\dots,[C_{\rho(S)}])$ and $([D_1],\dots,[D_{\rho(S)}])$ are dual bases of $\CH^1(S)_\Q/{\sim_\mr{num}}$ with respect to the intersection pairing.
We get a decomposition $\fh(S)_\Q=\mbf1\oplus\mbf1(-1)^{\rho(S)}\oplus\mbf1(-2)\oplus M$ in which the summands $\mbf1$, $\mbf1(-1)^{\rho(S)}$, and $\mbf1(-2)$ are cut out by the projectors
\[
\epsilon^0\ceq\{x\}\times S,
\quad
\epsilon^2\ceq\sum_{i=1}^{\rho(S)}[C_i\times D_i],
\quad\text{and}\quad
\epsilon^4\ceq S\times\{x\},
\]
respectively, and $M$ is cut out by the ``complementary'' projector $[\Delta_S]-\epsilon^0-\epsilon^2-\epsilon^4$.
By construction, for any algebraically closed field $k'/k$, $\CH^*(M_{k'})$ is the subgroup of numerically trivial cycles on $S_{k'}$.
Thus, by \ref{item:dDOD} and the previous paragraph, $\CH^*(M_{k'})=0$, so $M=0$ by \cite[Lemma 2.1]{gg}.
We deduce that $\rho(S)=b_2(S)$, and therefore that the map $\mr{NS}(S)\otimes\Z_\ell\to H^2(S,\Z_\ell(1))$ is an isomorphism for all primes $\ell\neq p$.
(The cokernel is, by the Kummer sequence, the Tate module $T_\ell(\Br(S))$, which is torsion-free.)
By Poincar\'e duality for $\ell$-adic cohomology, the intersection pairing on $\CH^1(S)_{\Z[1/p]}/{\sim_\mr{num}}$ is perfect.
Therefore, we may assume that the $C_i$ and $D_i$ above form dual bases of $\CH^1(S)_{\Z[1/p]}/{\sim_{\mr{num}}}$.
Then we obtain the decomposition described in \ref{item:dDECOMP}.
Statement \ref{item:dINV} is clear.
Statement \ref{item:dWDS} follows from Remark \ref{rmk:aocIIISC} (but can also be proven directly at the level of cycles).
\end{proof}

\begin{thing}
If $\charac(k)=0$ and $S$ is a surface over $k$ satisfying $b_1(S)=0$ and $b_2(S)=\rho(S)$, then we expect that $\CH_0(S)\cong\Z$ according to a classical conjecture of Bloch (\cite{bloch}).
In a recent preprint (\cite[Theorems A and B]{guletskii}), Guletski\u\i\ announced a proof of this conjecture, together with the following positive-characteristic version:
If $\charac(k)>0$, $S$ has a model over $\ol{\F_p}$, $b_1(S)=0$, and $b_2(S)=\rho(S)$, then $\CH_0(S)\cong\Z$.
Previously, Bloch's conjecture was proven for all surfaces not of general type (\cite{bkl}) and several classes of surfaces of general type (due to many authors, starting with \cite{im}; see e.g.\ the references in \cite{pw} and \cite{guletskii}).
\end{thing}

\begin{rmk}
In \cite[Proposition 3.6(1)]{sy} (see also \cite[Proposition 2.14(2)]{sy}), it is incorrectly stated that the map $\CH^2(M^\mr{t}\otimes M')\to\bigoplus_{i=2,3}\Hom(H^i(M),H^i(M'))$ is injective whenever $S$ and $S'$ are as in Proposition \ref{prop:dod}.
The issue is that the first isomorphism written on the top of page 14 of \tit{loc.\ cit.\ }is false in general, because some of the torsion in the cohomology of a product is in the image of the K\"unneth map.
In particular, the answer to \cite[Problem 3.8]{sy} is ``no''.
This does not effect the other results of \tit{loc.\ cit.\ }(but ``eff'' must be replaced by ``nor'' in the exact sequence \cite[equation (7.4)]{sy}).
\end{rmk}

\section{The motive of an Enriques surface}
\label{sec:enriques-motives}

\noindent
Continue to assume that $k$ is algebraically closed, and let $p$ denote the exponential characteristic of $k$, but now additionally assume that $p\neq2$.

\begin{thing}
\label{thing:hom}
Let $S$ be an Enriques surface over $k$.
Recall that the $2$-adic cohomology groups of $S$ are $(\Z_2,0,\Z_2^{10}\oplus\Z/2,\Z/2,\Z_2)$, and the $\ell$-adic cohomology of $S$ is torsion-free for any prime $\ell\nmid2p$ (\cite[\S1.4]{enriques-I}).
By \cite{bkl}, the equivalent conditions of Proposition \ref{prop:dod} hold for $S$; in particular, the torsion summand $\ft(S)$ of $\fh(S)_{\Z[1/p]}$ exists, and it is $2$-torsion by Lemma \ref{lem:action-on-coh}\ref{item:aocTO}.
Now let $S'$ be another Enriques surface over $k$.
Then \eqref{thing:tis-hom} yields a split exact sequence
\[
0
\to
\Z/2
\to
\Hom(\ft(S),\ft(S'))
\xrightarrow{\tau}
\bigoplus_{i=2}^3
\Hom(H^i(S,\Z_2)_\tors,H^i(S',\Z_2)_\tors),
\]
and the last term can be identified with $\Z/2\oplus\Z/2$.
Let $I(S,S')$ denote the image of $\tau$, viewed as a subgroup $\Z/2\oplus\Z/2$.
Notice that $\Coker(\tau)=H^4(S\times S',\Z_2)_\mr{tors}/(\text{algebraic classes})$.
\end{thing}

\begin{lem}
\label{lem:c}
Let $S$ and $S'$ be Enriques surfaces over $k$.
In the notation of \tn(\ref{thing:hom}\tn), we have that
\begin{enumerate}
\item
\label{item:cISO}
$\ft(S)\cong\ft(S')$ if and only if $(1,1)\in I(S,S')$.
\item
\label{item:cINDEC}
$\ft(S)$ is indecomposable if and only if
$I(S,S)=\{(0,0),(1,1)\}$.
\item
\label{item:cDEC}
$\ft(S)$ is decomposable if and only if $I(S,S)=\Z/2\oplus\Z/2$.
If this is the case, $\ft(S)=\ft^2(S)\oplus\ft^3(S)$ for nontrivial motives $\ft^2(S)$ and $\ft^3(S)$ which are unique up to isomorphism.
They satisfy $H^*(\ft^i(S),\Z_2)=H^i(S,\Z_2)_\tors$, $
\CH^*(\ft^2(S))
=
\CH^1(S)_\tors$, $
\CH^*(\ft^3(S))
=
0$, and $\ft^2(S)^{\mr t}\cong\ft^3(S)$.
\end{enumerate}
\end{lem}

\begin{proof}
Statement \ref{item:cISO} follows from Lemma \ref{lem:action-on-coh}\ref{item:aocIIIS}.
Noting that $(1,1)=\tau(\id_{\ft(S)})\in I(S,S)$, statement \ref{item:cINDEC} and the first sentence of \ref{item:cDEC} follow from Lemma \ref{lem:action-on-coh}\ref{item:aocOI}.
The computation of $H^*(\ft^i(S),\Z_2)$ and $\CH^*(\ft^i(S))$ is easy, and the former implies the uniqueness of $\ft^2(S)$ and $\ft^3(S)$ as well as the isomorphism $\ft^2(S)^{\mr t}\cong\ft^3(S)$ by Remark \ref{rmk:aocIIISC}.
\end{proof}

\begin{rmk}
\label{rmk:sycomp}
Since $\CH^2(S_{k(S')})$ is the direct limit of $\CH^2(S\times U)$, where $U$ ranges over the nonempty open subvarieties of $S'$, we have $I(S,S')\cong\CH_0(S_{k(S')})$ by the localization sequence for Chow groups.
In particular, if $\ft(S)$ is decomposable, then $\CH_0(S_{k(S)})=\Z/2\oplus\Z/2$.
In turn, $H^3_\mr{ur}(S\times S,\Q_\ell/\Z_\ell(2))=0$ for any prime $\ell\neq p$ by \cite[Corollary 6.4(a)]{kahn}.
\end{rmk}

\noindent
We will show in Examples \ref{eg:c-possibilities} and \ref{eg:c-possibilities-2} that, when $\ft(S)\ncong\ft(S')$, all three possibilities for $I(S,S')$ do occur.
For now, we turn toward the proofs of Theorems \ref{thm:enriques-thm-1} and \ref{thm:enriques-thm-2}.
The former will be a consequence of the following more general fact.

\begin{prop}
\label{prop:enriques-rat-map}
Let $f\cln S\dashrightarrow S'$ be a rational map between Enriques surfaces over $k$.
\begin{enumerate}
\item\label{item:ermISO}
Assume $f$ is of odd degree.
Then $\ft(S)\cong\ft(S')$.
\item\label{item:ermDECOMP}
Assume $f$ is of even degree.
Then $(0,1)\in I(S,S')$.
In particular, if $S=S'$, then $\ft(S)$ is decomposable.
\end{enumerate}
\end{prop}

\begin{proof}
Resolving $f$ gives a diagram \smash{$S\overset\pi\leftarrow\tilde S\overset{\tilde f}\to S'$} in which $\pi$ is a composition of blowups and $\tilde f$ is a generically finite morphism of odd degree.
By the blowup formula \eqref{thing:blowup}, the map $\pi^*\cln H^*(S,\Z_2)_\tors\to H^*(\tilde S,\Z_2)_\tors$ is an isomorphism, and by Poincar\'e duality, so is $\pi_*$.
In particular, $H^i(\tilde S,\Z_2)_\tors\cong\Z/2$ for $i\in\{2,3\}$.

\ref{item:ermISO}.
Since $\tilde f$ is generically finite of odd degree, the projection formula shows that $\tilde f_*\circ\tilde f^*$ acts on cohomology as multiplication by an odd number.
We deduce that $\tilde f_*\circ\pi^*$ induces an isomorphism $H^*(S,\Z_2)_\tors\iso H^*(S',\Z_2)_\tors$.
Therefore, the image of $\tilde S$ in $S\times S'$ defines an element of $\CH^2(S\times S')$ whose image in $\CH^2(\ft(S)^\vee\otimes\ft(S'))$ induces an isomorphism $H^*(\ft(S),\Z_2)\iso H^*(\ft(S'),\Z_2)$.
We conclude the proof by applying Lemma \ref{lem:action-on-coh}\ref{item:aocIIIS}.

\ref{item:ermDECOMP}.
Consider the fiber product
\[
\begin{tikzcd}
\tilde X\arrow[r]\arrow[d]&
X'\arrow[d]\\
\tilde S\arrow[r,"\tilde f"]&
S',
\end{tikzcd}
\]
where $X'\to S'$ is K3 cover.
If $\tilde X$ were disconnected, then $\tilde f$ would factor through a surjective morphism $\tilde S\to X'$; but this is impossible, because $\tilde S$ would inherit a nontrivial holomorphic $2$-form from $X$.
Therefore, $\tilde f^*\cln H^1(S',\Z/2)\to H^1(\tilde S,\Z/2)$ is injective.
Since $H^1(S',\Z_2)\cong H^1(\tilde S,\Z_2)\cong0$ by the blowup formula, the boundary morphisms $H^1(S',\Z/2)\to H^2(S',\Z_2)_\tors$ and $H^1(\tilde S,\Z/2)\to H^2(\tilde S,\Z_2)_\tors$ are isomorphisms.
Thus $\tilde f^*\cln H^2(S',\Z_2)_\tors\to H^2(\tilde S,\Z_2)_\tors$ is also an isomorphism.
By Poincar\'e duality, so is $\tilde f_*\cln H^3(\tilde S,\Z_2)_\tors\to H^3(S',\Z_2)_\tors$.
Moreover, $\tilde f_*\circ\tilde f^*$ induces multiplication by an even number (hence zero) on $H^2(S',\Z_2)_\tors$, so the map $\tilde f_*\cln H^2(\tilde S,\Z_2)_\tors\to H^2(S',\Z_2)_\tors$ must vanish.
All in all, $\tilde f_*\circ\pi^*$ kills $H^2(S,\Z_2)_\tors$ but acts as an isomorphism on $H^3(S',\Z_2)_\tors$.
We conclude, as in the proof of \ref{item:ermISO}, by considering the image of $\tilde S$ in $S\times S'$; the statement for $S=S'$ follows from Lemma \ref{lem:action-on-coh}\ref{item:aocOI}.
\end{proof}

\begin{proof}[Proof of Theorem \ref{thm:enriques-thm-1}]
Let $\circ,\bullet\in\mf L$.
If $\phi\cln E_\circ\to E_\bullet$ and $\psi\cln F_\circ\to F_\bullet$ are isogenies satisfying $\phi(P_\circ)=P_\bullet$ and $\psi(Q_\circ)=Q_\bullet$, then $\phi$ and $\psi$ induce a rational map $S_\circ\dashrightarrow S_\bullet$ of degree $\deg(\phi)\cdot\deg(\psi)$.
Therefore, statement \ref{item:et1ISO} of the present theorem follows immediately from Proposition \ref{prop:enriques-rat-map}\ref{item:ermISO}.

\ref{item:et1DECOMP}.
Recall that $E_\circ$ is an elliptic curve over $k$ such that $\End(E_\circ)\cong\cO_L$, where $L$ is a quadratic imaginary number field in which $2$ splits.
By the previous paragraph and Proposition \ref{prop:enriques-rat-map}\ref{item:ermDECOMP}, it suffices to find $P_\circ\in E_\circ[2]\setminus\{O\}$ and an endomorphism $\phi\in\End(E_\circ)$ of even degree satisfying $\phi(P_\circ)=P_\circ$.
By hypothesis, $E_\circ[2]$ is isomorphic, as an $\End(E_\circ)\cong\mc O_L$-module, to $\mc O_L/2\cong\mc O_L/\mf p\times\mc O_L/\ol{\mf p}\cong\F_2\times\F_2$, where $2\mc O_L=\mf p\cdot\ol{\mf p}$ is the factorization into prime ideals.
Thus, if we pick $P_\circ$ and $\phi$ whose images in $\F_2\times\F_2$ (under the maps
\[
E_\circ[2]\cong\cO_L/2\cong\F_2\times\F_2
\quad\text{and}\quad
\End(E_\circ)\cong\cO_L\onto\cO_L/2\cong\F_2\times\F_2,
\]
respectively) coincide and equal $(1,0)$ or $(0,1)$, then $\deg(\phi)$ is even and $\phi(P_\circ)=P_\circ$.
\end{proof}


\begin{proof}[Proof of Theorem \ref{thm:enriques-thm-2}]
Let's assume that $(\lambda,\mu,\lambda'^*)$ is suitable; the other cases are similar.
Let $\alpha$ and $\alpha'$ denote the unique nonzero elements in $H^1(S_{\lambda,\mu},\mu_2)$ and $H^1(S_{\lambda',\mu'},\mu_2)$, respectively, and let $\alpha_{P_{\lambda'^*}}\in H^1(E_{\lambda'^*},\mu_2)$ be as in \ref{thing:ec-coh}\ref{item:ec2}.
Note that $H^1(\ft(S_{\lambda,\mu}),\mu_2)=H^1(S_{\lambda,\mu},\mu_2)$, and similarly for $S_{\lambda',\mu'}$ (and also for $H^3$).

There exists a pullback diagram of the form
\[
\begin{tikzcd}
E_{\lambda'}
\arrow[r]
\arrow[d]
&
\Kum(E_{\lambda'}\times E_{\mu'})
\arrow[d]
\\
E_{\lambda'}/P_{\lambda'}
\arrow[r,"f"]
&
S_{\lambda',\mu'}.
\end{tikzcd}
\]
(This $f$ is the inclusion of a half-fiber of an elliptic fibration on $S_{\lambda',\mu'}$.)
As noted in \eqref{thing:star}, we have $E_{\lambda'}/P_{\lambda'}\cong E_{\lambda'^*}$ and, under this isomorphism, $f^*\alpha'=\alpha_{P_{\lambda'^*}}$.
In particular, $f$ induces a map $\ft(S_{\lambda',\mu'})\to\fh(E_{\lambda'^*})$ sending $\alpha'$ to $\alpha_{P_{\lambda'^*}}$.

Now let $\beta$ be any non-algebraic class in $H^2(S_{\lambda,\mu},\mu_2)$, and let $\gamma$ be the unique nonzero class in $H^3(S_{\lambda,\mu},\mu_2)$.
Then $\gamma=\ol\partial\beta$, where $\ol\partial$ is the boundary map $H^2(S_{\lambda,\mu},\mu_2)\to H^3(S_{\lambda,\mu},\mu_2)$, i.e.\ the composition of the boundary map $\partial\cln H^2(S_{\lambda,\mu},\mu_2)\to H^3(S_{\lambda,\mu},\Z_2(1))$ with reduction mod $2$.
Since $\partial\alpha_{P_{\lambda'^*}}=0$, the Leibniz rule gives $\ol\partial(\beta\otimes\alpha_{P_{\lambda'^*}})=\gamma\otimes\alpha_{P_{\lambda'^*}}$.
Thus the class $\gamma\otimes\alpha_{P_{\lambda'^*}}\in H^4(S_{\lambda,\mu}\times E_{\lambda'^*},\mu_2^{\otimes2})$ is not algebraic by Corollary \ref{cor:galois-nonzero-application}, Proposition \ref{prop:ur}, and the assumption that $(\lambda,\mu,\lambda')$ is suitable.
Equivalently, no correspondence $\Gamma\in\CH^2(S_{\lambda,\mu}\times E_{\lambda'^*})$ satisfies $\Gamma_*\alpha=\alpha_{P_{\lambda'^*}}$ (see \eqref{eq:kp0}).
Therefore, there is no map $\ft(S_{\lambda,\mu})\to\fh(E_{\lambda'^*})$ sending $\alpha$ to $\alpha_{P_{\lambda'^*}}$.
Comparing with the final sentence of the previous paragraph, we see that $\ft(S_{\lambda,\mu})\ncong\ft(S_{\lambda',\mu'})$, as desired.
\end{proof}

\begin{eg}
\label{eg:c-possibilities}
Let $\lambda,\mu,\lambda',\mu'\in k\setminus\{0,1\}$.
Using that the boundary map $H^1(\ft(S),\Z/2)\to H^2(\ft(S),\Z_2)$ is an isomorphism for any Enriques surface $S$ over $k$, it follows from the proof of Theorem \ref{thm:enriques-thm-1} that if either $(\lambda,\mu,\lambda'^*)$ or $(\lambda,\mu,\mu'^*)$ is suitable, then $I(S_{\lambda,\mu},S_{\lambda',\mu'})\subseteq0\oplus\Z/2$ and by Poincar\'e duality, $I(S_{\lambda',\mu'},S_{\lambda,\mu})\subseteq\Z/2\oplus0$.
Therefore, if additionally one of $(\lambda',\mu',\lambda^*)$ or $(\lambda',\mu',\mu^*)$ is suitable, then $I(S_{\lambda,\mu},S_{\lambda',\mu'})=0$.
For example, if $k=\ol\Q$, $\lambda=\mu=25$, and $\lambda'=\mu'=25^*=9/4$, then \ref{defn:suitable}\ref{item:gnaBB} holds for $(\lambda,\mu,\lambda'^*)$ with $k_0=\Q$ and $v$ the $5$-adic valuation and for $(\lambda',\mu',\lambda^*)$ with $k_0=\Q$ and $v$ the $3$-adic valuation.
\end{eg}

\begin{eg}
\label{eg:c-possibilities-2}
Let $\lambda,\mu\in k\setminus\{0,1\}$.
Under the obvious isomorphism $E_{1-\lambda}\iso E_\lambda$, the point $P_{1-\lambda}$ does not map to $P_\lambda$.
That way, by \eqref{thing:star}, there is a degree-$2$ isogeny $E_{1-\lambda}\to E_{\lambda^*}$ which sends $P_{1-\lambda}$ to $P_{\lambda^*}$.
As in the proof of Theorem \ref{thm:enriques-thm-1}, we obtain a degree-$2$ rational map $S_{1-\lambda,\mu}\dashrightarrow S_{\lambda^*,\mu}$.
If, in addition, $(1-\lambda,\mu,\lambda)$ is suitable, then $I(S_{1-\lambda,\mu},S_{\lambda^*,\mu})=0\oplus\Z/2$ and $I(S_{\lambda^*,\mu},S_{1-\lambda,\mu})=\Z/2\oplus0$ by the previous paragraph and part \ref{item:ermDECOMP} of Proposition \ref{prop:enriques-rat-map}.
For example, if $k=\ol\Q$, $\lambda=4/9$ (so $\lambda^*=1/25$), and $\mu=-1$, then \ref{defn:suitable}\ref{item:gnaBB} holds for $(1-\lambda,\mu,\lambda)$ with $k_0=\Q$ and $v$ the $3$-adic valuation.
\end{eg}

\begin{eg}
\label{eg:nsni}
Assume that $S$ and $S'$ are Enriques surfaces over $k$ such that $I(S,S')=0\oplus\Z/2$ and $\ft(S)$ and $\ft(S')$ are decomposable.
Then the nonzero map $\ft(S)\to\ft(S')$ induces a map $\ft^3(S)\to\ft^3(S')$ which is an isomorphism on cohomology.
However, it is not possible that $\ft^3(S)\cong\ft^3(S')$, for this would imply that $\ft^2(S)\cong\ft^2(S')$ (since $\ft^2$ is dual to $\ft^3$), hence in turn that $\ft(S)\cong\ft(S')$ (since $\ft=\ft^2\oplus\ft^3$).

Explicitly, consider the curve $E$ over $\ol\Q$ given by $y^2=(x+7)(x^2-7x+14)$, which satisfies $\End(E)\cong\mc O_{\Q(\sqrt{-7})}$ (\cite[elliptic curve \href{https://www.lmfdb.org/EllipticCurve/Q/784/f/4}{\texttt{784.f4}}]{lmfdb}).
Let $\mu\ceq(e_2+7)/(e_3+7)$, where $e_2$ and $e_3$ are the roots of $x^2-7x+14$;
then $E\cong E_\mu$.
Since $e_2$ and $e_3$ are conjugate over $\Q$ and $3$ is inert in $\Q(\sqrt{-7})$, we have $v(\mu)=0$ for the $3$-adic valuation $v$ on $\Q(\sqrt{-7})$.
Therefore $I(S_{-5/9,\mu},S_{1/25,\mu})=0\oplus\Z/2$ as in Example \ref{eg:c-possibilities-2}, while $\ft(S_{-5/9,\mu})$ and $\ft(S_{1/25,\mu})$ are decomposable by part \ref{item:et1DECOMP} of Theorem \ref{thm:enriques-thm-1}.
\end{eg}

\begin{eg}
\label{eg:same-cover}
We can also construct examples where the Enriques surfaces share the same K3 cover.
For example, if $\lambda,\mu\in k\setminus\{0,1\}$ and $(\lambda,1-\mu,\mu^*)$ and $(1-\lambda,\mu,\lambda^*)$ are both suitable, then $I(S_{\lambda,1-\mu},S_{1-\lambda,\mu})=0$, while $S_{\lambda,1-\mu}$ and $S_{1-\lambda,\mu}$ have isomorphic K3 covers (since $E_\lambda\cong E_{1-\lambda}$ and $E_\mu\cong E_{1-\mu}$).
Explicitly, if we take $k=\ol\Q$, $\lambda=25$ (so $\lambda^*=9/4$), and $\mu=36$ (so $\mu^*=49/25$), then $(\lambda,1-\mu,\mu^*)$ satisfies \ref{defn:suitable}\ref{item:gnaBB} with $v$ the $5$-adic valuation on $\Q$, and $(1-\lambda,\mu,\lambda^*)$ satisfies \ref{defn:suitable}\ref{item:gnaBB} with $v$ the $3$-adic valuation on $\Q$.
\end{eg}

\appendix

\section{On \texorpdfstring{\cite[Proposition 4.21]{vishik}}{[Vis23, Proposition 4.21]}}
\label{appx:vishik}

\noindent
The error in the proof of \cite[Proposition 4.21]{vishik} arises in the phrase ``this pair of Rost submodules satisfies all the conditions of Theorem 4.20(2)''.
In fact, condition (a) of \cite[Theorem 4.20(2)]{vishik} fails, as we will explain in \eqref{thing:mistake} below.

Let $k$ be an algebraically closed field of characteristic $0$.

\begin{thing}
For a ring $R$, let $\DM(k,R)$ denote Voevodsky's triangulated category of motives over $k$ with coefficients in $R$, and let $\DMgm(k,R)$ be the subcategory of geometric motives.
For $M\in\DM(k,R)$, its motivic cohomology and homology are given by $H_\mot^{i,j}(M,R)\ceq\Hom_{\DM(k,R)}(M,R(j)[i])$ and $H^\mot_{i,j}(M,R)\ceq\Hom_{\DM(k,R)}(R(j)[i],M)$, respectively (where, importantly, we take morphisms in the category with $R$-coefficients).
\end{thing}

\begin{thing}
\label{thing:dmcoc}
For any integer $n\geq0$, the functor from the category of smooth projective $k$-varieties to $\DM(k,\Z/n)$ factors through a fully faithful functor $\CHM(k,\Z/n)^\mr{op}\to\DMgm(k,\Z/n)$.
This is well-known when $n=0$ (e.g.\ \cite[Proposition 20.1]{mvw});
for $n\geq1$, the same proof works once one checks that $H_\mot^{2i,i}(M(X),\Z/n)\cong\CH^i(X)\otimes\Z/n$, which can be done using the following adjoint triple of functors $\nu^*\dashv\nu_*\dashv\nu^*[-1]$:
\[
\DM(k,\Z)
\xrightarrow{\nu^*}
\DM(k,\Z/n)
\xrightarrow{\nu_*}
\DM(k,\Z)
\xrightarrow{\nu^*[-1]}
\DM(k,\Z/n),
\]
where $\nu^*(M)\ceq\mr{Cone}(n\cln M\to M)$ and $\nu_*$ views a motive with $\Z/n$-coefficients as a motive with $\Z$-coefficients.
\end{thing}

\noindent
The following application of Bondarko's theory of weight structures (\cite{bondarko}) allows one to detect the Chow motives in $\mbf{DM}_\mr{gm}(k,\Z/n)$ using motivic cohomology and homology.

\begin{prop}
\label{prop:bondarko}
Fix $n\geq0$, and let $M\in\DM_\mr{gm}(k,\Z/n)$.
Then the fol{}lowing are equivalent:
\begin{enumerate}
\item
$M$ is a Chow motive with $\Z/n$-coefficients, i.e.\ lives in the essential image of the functor $\CHM(k,\Z/n)\to\DM_\mr{gm}(k,\Z/n)$.
\item
For any finitely generated field extension $k'/k$, we have
$H_\mot^{i,j}(M_{k'},\Z/n)=0$ for al{}l $i>2j$, and $H^\mot_{i,j}(M_{k'},\Z/n)=0$ for al{}l $i<2j$.
\end{enumerate}
\end{prop}

\begin{proof}
When $n=0$, this is \cite[Proposition 5.1]{vishik}.
The same proof works for $n\geq1$, using Bondarko's Chow weight structure on $\DM_\mr{gm}(k,\Z/n)$ instead (\cite[Remark 6.6.1]{bondarko}).
\end{proof}

\begin{cor}
\label{cor:nncm}
If $n\geq1$, then no nontrivial Chow motive with $\Z$-coefficients is of the form $\nu_*(N)$ for some $N\in\DM_\mr{gm}(k,\Z/n)$.
\end{cor}

\begin{proof}
Let $N$ be as above, and assume that its image $M\ceq\nu_*(N)$ in $\DM(k,\Z)$ is a Chow motive (with $\Z$-coefficients).
Using the adjoint triple of functors $\nu^*\dashv\nu_*\dashv\nu^*[-1]$ of \eqref{thing:dmcoc}, we find that $H^\mot_{i,j}(N,\Z/n)\cong H^\mot_{i,j}(M,\Z)$ and $H_\mot^{i,j}(N,\Z/n)\cong H_\mot^{i+1,j}(M,\Z)$ for all $i,j$.
Thus, since $M$ is a Chow motive, $N$ is also a Chow motive (with $\Z/n$-coefficients) by Proposition \ref{prop:bondarko}.
But also $\CH^i(N)\cong H_\mot^{2i,i}(N,\Z/n)\cong H_\mot^{2i+1,i}(M,\Z)=0$ for all $i$, and varying the base field $k$, the same argument gives that $\CH^i(N_{k'})=0$ for all $i$ and any field extension $k'/k$.
This implies that $N=0$, hence $M=0$.
\end{proof}

\begin{thing}
\label{thing:mistake}
Now we explain in detail what goes wrong in \cite[Proposition 4.21]{vishik}.
In the notation there, we claim that $N_\bullet(C)$ is not compact, i.e.\ not a geometric motive (\cite[Theorem 11.1.13]{cd}), contrary to condition (a) of \cite[Theorem 4.20(2)]{vishik}.
Indeed, if it were a geometric motive, the proof of \cite[Theorem 4.20]{vishik} would produce an exact triangle
\[
\mc T_C(-1)\to N\to N_\bullet(C)\xrightarrow{+1}
\]
in which $N\in\DMgm(k,\Z/n)$ (and this $N$ appears in the motive of a smooth projective curve).
Moreover, tracing through the proof shows that application of $\nu_*$ to this exact triangle yields the exact triangle
\[
\mc T_A(-1)\to M\to M_\bullet(A)\xrightarrow{+1}
\]
defining $M_\bullet(A)$.
(Note that $\nu_*(N_\bullet(C))=M_\bullet(A)$ and $\nu_*(\mc T_C)=\mc T_A$ by definition, and similarly for the motives $\mc K_C$ and $\mc K_A$ used in \cite{vishik}.)
We deduce from Corollary \ref{cor:nncm} that $M=0$, contrary to the hypothesis of \cite[Proposition 4.21]{vishik}.
\end{thing}

\section{Poincar\'e and K\"unneth in the presence of torsion}
\label{appx:poincare-kunneth}

\noindent
Throughout, $k$ is an algebraically closed field, $X$ and $Y$ are connected smooth projective varieties over $k$ of respective dimensions $d_X$ and $d_Y$, and $\ell$ is a fixed prime number not equal to $\charac(k)$.
For simplicity, we use a fixed compatible system of roots of unity to identify all Tate twists in \'etale cohomology.
However, everything can be adapted to work for general $k$, in which case all the maps (with the appropriate twists) become equivariant for the Galois action.
Also, everything below has analogues for Betti and crystalline cohomology, and in the latter case, everything respects the Frobenius actions.
The material in this section is well-known to experts, but we could not find some of it, especially Proposition \ref{prop:good-kunneth}, in the literature.

We begin with some basic applications of Poincar\'e duality in \'etale cohomology.

\begin{thing}
\label{thing:cohfin}
Fix a positive integer $n$.
There exists a ``trace'' isomorphism $\tr_X\cln H^{2d_X}(X,\Z/\ell^n)\iso\Z/\ell^n$ which makes the cup-product pairing
\[
\gen{-,-}_n
\cln
H^i(X,\Z/\ell^n)
\times
H^{2d_X-i}(X,\Z/\ell^n)
\to
H^{2d_X}(X,\Z/\ell^n)
\iso
\Z/\ell^n
\]
perfect for each $i$;
this is a version of Poincar\'e duality for $X$ in \'etale cohomology.
Therefore, for a morphism $f\cln X\to Y$, we have a ``pushforward'' map
$
f_*
\cln
H^i(X,\Z/\ell^n)
\to
H^{i+2(d_Y-d_X)}(Y,\Z/\ell^n)
$
which is defined to be equal, under Poincar\'e duality, to $(f^*)^\vee$, where $(-)^\vee$ denotes the functor $\Hom(-,\Z/\ell^n)$.
In other words, $f_*$ is defined by the identity $\gen{f_*\alpha,\beta}_n=\gen{\alpha,f^*\beta}_n$.
Then, for any $\Gamma\in H^{2d_X+e}(X\times Y,\Z/\ell^n)$, we define
$
\Gamma_*
\cln
H^i(X,\Z/\ell^n)
\to
H^{i+e}(Y,\Z/\ell^n)
$
by $\Gamma_*\alpha\ceq\mr{pr}_{2*}(\mr{pr}_1^*\alpha\cdot\Gamma)$, where $\mr{pr}_1$ and $\mr{pr}_2$ denote the projection maps $X\times Y\to X$ and $X\times Y\to Y$, respectively.

\end{thing}

\begin{thing}
Everything in the previous paragraph is compatible with the reduction maps $\Z/\ell^{n+1}\to\Z/\ell^n$, so we can define pushforward maps on $H^*(-,\Z_\ell)$ as well as maps of the form
\begin{equation}
\label{eq:zlpf}
\Gamma_*
\cln
H^i(X,\Z_\ell)
\to
H^{i+e}(Y,\Z_\ell)
\end{equation}
for any $\Gamma\in H^{2d_X+e}(X\times Y,\Z_\ell)$.
Also, we have $H^{2d_X}(X,\Z_\ell)\iso\Z_\ell$ and therefore a pairing
\[
\gen{-,-}
\cln
H^i(X,\Z_\ell)
\times
H^{2d_X-i}(X,\Z_\ell)
\to
H^{2d_X}(X,\Z_\ell)
\iso
\Z_\ell,
\]
which is perfect modulo torsion.
\end{thing}

\begin{thing}
\label{thing:torsion-pairing}
There is also a less well-known pairing, for any positive integer $n$, of the form
\begin{equation}
\label{eq:tp}
\ggen{-,-}_n
\cln
H^i(X,\Z_\ell)[\ell^n]
\times
H^{2d_X-i+1}(X,\Z_\ell)[\ell^n]
\to
\Z/\ell^n.
\end{equation}
It is given by $\ggen{\alpha,\beta}_n\ceq\gen{\ol\alpha,\tilde\beta}_n$, where $\ol\alpha$ is the mod-$\ell^n$ reduction of $\alpha$ and $\tilde\beta$ is any lift of $\beta$ along the boundary map $\partial\cln H^{2d_X-i}(X,\Z/\ell^n)\to H^{2d_X-i+1}(X,\Z_\ell)$.
This pairing is easily checked to be well-defined.
It is not in general perfect, but its left and right kernels are
\begin{align*}
&P^i_{\ell^n}(X)
&&\hspace{-6em}\ceq
H^i(X,\Z_\ell)[\ell^n]\cap \ell^nH^i(X,\Z_\ell)
\qquad\qquad\qquad\quad\text{and}
\\
&P^{2d_X-i+1}_{\ell^n}(X)
&&\hspace{-6em}\ceq
H^{2d_X-i+1}(X,\Z_\ell)[\ell^n]\cap \ell^nH^{2d_X-i+1}(X,\Z_\ell),
\end{align*}
respectively, which vanish for $n$ sufficiently large.
That is, $\ggen{-,-}_n$ descends to a perfect pairing $Q^i_{\ell^n}(X)\times Q^{2d_X-i+1}_{\ell^n}(X)\to\Z/\ell^n$, where $Q^i_{\ell^n}(X)\ceq H^i(X,\Z_\ell)[\ell^n]/P^i_{\ell^n}(X)$.
(In other words, $Q^i_{\ell^n}(X)$ is the image of $H^i(X,\Z_\ell)[\ell^n]$ in $H^i(X,\Z/\ell^n)$.)
One usually takes the union over all $n$ and considers the induced perfect pairing $H^i(X)_\mr{tors}\times H^{2d_X-i+1}(X)_\mr{tors}\to\Q/\Z$, but we use the above formulation so that the identity \eqref{eq:kp2} below makes sense.
\end{thing}

\noindent
Items \eqref{thing:cohfin}--\eqref{thing:torsion-pairing} have the following applications for Chow motives.

\begin{thing}
Fix a positive integer $a$ not divisible by $\ell$, and let $M\ceq(X,\epsilon,s)$ be a Chow motive with $\Z[1/a]$-coefficients; recall that this means $\epsilon\in\CH^{d_X}(X\times X)_{\Z[1/a]}$ is an idempotent correspondence and $s\in\Z$.
Let $\Lambda\in\{\Z/\ell,\Z/\ell^2,\dots,\Z_\ell\}$.
We define
\[
H^i(M,\Lambda)
\ceq
\mr{cl}_\Lambda(\epsilon)_*H^{i-2a}(X,\Lambda),
\]
where $\mr{cl}_\Lambda$ denotes the cycle-class map $\CH^i(-)_{\Z[1/a]}\to H^{2i}(-,\Lambda)$.
If $N\ceq(Y,\eta,t)$ is another Chow motive and $\phi\in\Hom(M,N)$, i.e.\ $\phi\in\eta\circ\CH^{2d_X+(t-s)}(X\times Y)_{\Z[1/a]}\circ\epsilon$, then $\cl_\Lambda(\phi)_*$ defines a map $H^i(M,\Lambda)\to H^i(N,\Lambda)$.
This defines a functor $H^i(-,\Lambda)$ on the category of Chow motives.
One checks that not only is $H^i(M,\Lambda)$ functorial in the ring $\Lambda$, but also that a short exact sequence $0\to\Lambda_1\to\Lambda_2\to\Lambda_3\to0$ induces a long exact sequence
\[
\cdots
\to
H^i(M,\Lambda_1)
\to
H^i(M,\Lambda_2)
\to
H^i(M,\Lambda_3)
\to
H^{i+1}(M,\Lambda_1)
\to
\cdots
\]
which is functorial in $M$ and the short exact sequence.
\end{thing}

\begin{thing}
\label{thing:motive-pairing}
Let $(-)^\mr{t}$ denote the map $H^*(X\times Y,\Z/\ell^n)\to H^*(Y\times X,\Z/\ell^n)$ which is pullback along the ``swap'' isomorphism $Y\times X\to X\times Y$ (equivalently, pushforward along its inverse).
Then
\[
\gen{\Gamma_*\alpha,\beta}_n
=
(-1)^{e(i+e)}\cdot\gen{\alpha,\Gamma^\mr{t}_*\beta}_n
\]
for each $\Gamma\in H^{2d_X+e}(X\times Y,\Z/\ell^n)$, $\alpha\in H^i(X,\Z/\ell^n)$, and $\beta\in H^{2d_Y-(i+e)}(Y,\Z/\ell^n)$.
One deduces that the same formulas hold for the pairings $\gen{-,-}$ and $\ggen{-,-}_n$.
In particular, if $M\ceq(X,\epsilon,0)$ is a Chow motive and $M^{\mr t}\ceq(X,\epsilon^\mr{t},0)$, then there are induced pairings
\[
\begin{aligned}
&H^i(M,\Z/\ell^n)
&\times\hspace{1em}
&H^{2d_X-i}(M^{\mr t},\Z/\ell^n)
&\to\hspace{1em}
&\Z/\ell^n,
\\
&H^i(M,\Z_\ell)
&\times\hspace{1em}
&H^{2d_X-i}(M^{\mr t},\Z_\ell)
&\to\hspace{1em}
&\Z_\ell,
\\
&H^i(M,\Z_\ell)[\ell^n]
&\times\hspace{1em}
&H^{2d_X-i+1}(M^{\mr t},\Z_\ell)[\ell^n]
&\to\hspace{1em}
&\Z/\ell^n;
\end{aligned}
\]
the first is perfect, the second is perfect modulo torsion, and the third has left- and right-kernels as described in \eqref{thing:torsion-pairing}.
\end{thing}

\noindent
Finally, we describe a version of the K\"unneth theorem for $\Z_\ell$-cohomology which gives an explicit description of the $\mr{Tor}$ terms.
First, we need the following property of the maps $\Gamma_*$.

\begin{thing}
The identity $\tr_{X\times Y}=\tr_X\cdot\tr_Y$ implies that
\begin{equation}
\label{eq:kp0}
(\alpha_0\boxtimes\beta_0)_*\alpha
=
\gen{\alpha_0,\alpha}_n\cdot\beta_0
\end{equation}
for any $\alpha_0\in H^i(X,\Z/\ell^n)$, $\beta_0\in H^{2d_X-i+e}(Y,\Z/\ell^n)$, and $\alpha\in H^j(X,\Z/\ell^n)$, where we put $\gen{\alpha_0,\alpha}_n\ceq0$ if $\alpha_0$ and $\alpha$ do not live in complementary dimensions.
By \eqref{eq:kp0}, we have
\begin{equation}
\label{eq:kp}
(\alpha_0\boxtimes\beta_0)_*\alpha
=
\gen{\alpha_0,\alpha}\cdot\beta_0
\end{equation}
for any $\alpha_0\in H^i(X,\Z_\ell)$, $\beta_0\in H^{2d_X-i}(Y,\Z_\ell)$, and $\alpha\in H^j(X,\Z_\ell)$.
There is also an analogue of the identity \eqref{eq:kp} for the pairings $\ggen{-,-}_n$ (which can be deduced from \eqref{eq:kp0}).
Let $\alpha_0\in H^i(X,\Z_\ell)[\ell^n]$ and $\beta_0\in H^{2d_X-i+1}(Y,\Z_\ell)[\ell^n]$, and again let $\tilde{\alpha_0}$ and $\tilde{\beta_0}$ be lifts along the appropriate boundary maps.
Then, for any $\alpha\in H^j(X,\Z_\ell)[\ell^n]$, we have
\begin{equation}
\label{eq:kp2}
[(\partial(\tilde{\alpha_0}\boxtimes\tilde{\beta_0}))_*\alpha]
=
\ggen{\alpha_0,\alpha}_n\cdot[\beta_0],
\end{equation}
where $\partial$ is the boundary map and the brackets denote the image in the group $Q_{\ell^n}^{j+e}(Y)$ defined in \eqref{thing:torsion-pairing}.
\end{thing}

\begin{prop}
\label{prop:good-kunneth}
Let $X$ and $Y$ be connected smooth projective varieties over $k$, and let $M$ and $N$ be summands of $\fh(X)_{\Z[1/a]}$ and $\fh(Y)_{\Z[1/a]}$, respectively.
Fix $e\in\Z$.
Then
\[
\begin{aligned}
0
\to
\bigoplus_{a+b=d_X+e}
H^a(M^{\mr t},\Z_\ell)
\otimes
H^b(N,\Z_\ell)
&\xrightarrow{\boxtimes}
H^{2d_X+e}(M^{\mr t}\times N,\Z_\ell)\\
&\hspace{1em}\xrightarrow{(-)_*}
\bigoplus_j
\Hom(
H^i(M,\Z_\ell)_\mr{tors},H^{i+e}(N,\Z_\ell)_\mr{tors}
)
\to
0
\end{aligned}
\]
is exact, where $(-)_*$ denotes the assignment $\Gamma\mapsto\Gamma_*$ of \tn(\ref{eq:zlpf}\tn).
\end{prop}

\begin{proof}
Because this sequence is compatible with the action of correspondences, we may and do assume that  $M=\fh(X)$ and $N=\fh(Y)$.
The exactness in this case is a consequence of the following four facts:
\begin{itemize}
\item
There is an exact sequence when we replace the final term by
\[
\bigoplus_{a+b=d_X+e+1}
\mr{Tor}_1^\Z(H^a(X,\Z_\ell)_\mr{tors},H^b(Y,\Z_\ell)_\mr{tors});
\]
this follows from the usual K\"unneth theorem in $\ell$-adic cohomology.
\item
The groups
\[
\bigoplus_i
\Hom(
H^i(X,\Z_\ell)_\mr{tors},H^{i+e}(Y,\Z_\ell)_\mr{tors}
)
\;\;\text{and}\;\;
\bigoplus_{a+b=d_X+e+1}
\mr{Tor}_1^\Z(H^a(X,\Z_\ell)_\mr{tors},H^b(Y,\Z_\ell)_\mr{tors})
\]
have the same cardinality; this follows from the fact that $\ggen{-,-}_n$ is perfect for $n$ sufficiently large.
\item
The sequence in question is a complex, i.e.\ $(\alpha_0\boxtimes\beta_0)_*\alpha=0$ for any torsion $\alpha$; this follows from the identity \eqref{eq:kp}.
\item
The map $(-)_*$ is surjective.
\end{itemize}
The last point is a bit more complicated to explain.
Fix $i\in\Z$, and set $A\ceq H^i(X,\Z_\ell)_\mr{tors}$ and $B\ceq H^{i+e}(Y,\Z_\ell)_\mr{tors}$ as well as $A_{\ell^n}\ceq Q_{\ell^n}^i(X)$ and $B_{\ell^n}\ceq Q_{\ell^n}^{i+e}(Y)$ in the notation of \eqref{thing:torsion-pairing}.
Let $H$ denote the image of $(-)_*$ in $\Hom(A,B)$.
Because $\ggen{-,-}_n\cln A_{\ell^n}\times B_{\ell^n}\to\Z/\ell^n$ is perfect, the formula \eqref{eq:kp2} implies that the image of $H$ in $\Hom(A_{\ell^n},B_{\ell^n})$ includes the image of the canonical map $\Hom(A_{\ell^n},\Z/\ell^n)\otimes B_{\ell^n}\to\Hom(A_{\ell^n},B_{\ell^n})$.
So, the surjectivity of $(-)_*$ is a special case of the following lemma.
\end{proof}

\begin{lem}
Let $A$ and $B$ be finite Abelian groups, and for each positive integer $n$, define
\[
A_n
\ceq
\frac{A[n]}{A[n]\cap nA}
\quad\text{and}\quad
B_n
\ceq
\frac{B[n]}{B[n]\cap nB}.
\]
Let $H$ be a subgroup of $\Hom(A,B)$ such that, for each integer $n$, the image of $H$ in $\Hom(A_n,B_n)$ includes the image of the canonical map
\[
\Phi_n
\cln
\Hom(A_n,\Z/n)\otimes B_n
\to
\Hom(A_n,B_n).
\]
Then $H=\Hom(A,B)$.
\end{lem}

\begin{proof}
Writing $A$ and $B$ as a direct sum of groups of prime-power order, we may and do assume that $A$ and $B$ are of $\ell$-power order for some fixed prime $\ell$.
Fix decompositions $A=\bigoplus_{i\geq1}A^i$ and $B=\bigoplus_{i\geq1}B^i$ where $A^i$ and $B^i$ are free over $\Z/\ell^i$.
Let us also write $A^{\leq N}\ceq\bigoplus_{i=1}^NA^i$, and likewise $B^{\leq N}$.
We will prove by induction on $N$ that the map (defined by the fixed decompositions)
\[
R_N
\cln
H
\to
\Hom(
A^{\leq N},
B^{\leq N}
)
\]
is surjective for all $N\geq0$.
The base case $N=0$ is trivial, so assume that $R_{N-1}$ is surjective for some $N\geq1$.
Fix $\phi\cln A^{\leq N}\to B^{\leq N}$; we aim to show that $\phi$ lifts to an element of $H$.
By the inductive hypothesis, we may and do assume that the map $A^{\leq N-1}\to B^{\leq N-1}$ induced by $\phi$ is zero.
Now consider the induced decompositions
\[
A_{\ell^N}=A^{\leq N-1}\oplus A^N\oplus\bigoplus_{i=1}^{N-1}\frac{\ell^iA^{N+i}}{\ell^N}
\eqqcolon
A(1)\oplus A(2)\oplus A(3)
\]
and similarly $B_{\ell^N}=B(1)\oplus B(2)\oplus B(3)$.
Then $\Phi_{\ell^N}$ factors as a direct sum of the canonical maps $\Phi_{i,j}\cln\Hom(A(i),\Z/\ell^N)\otimes B(j)\to\Hom(A(i),B(j))$ for $i,j\in\{1,2,3\}$.
In particular, since $A^N$ and $B^N$ are free over $\Z/\ell^N$, $\Phi_{i,j}$ is surjective if either $i=2$ or $j=2$.
This implies, since we already assumed that $\phi$ vanishes in $\Hom(A(1),B(1))$, that there is an element in the image of $\Phi_{\ell^N}$ whose image in $\Hom(A^{\leq N},B^{\leq N})=\bigoplus_{i,j=1}^2\Hom(A(i),B(j))$ is $\phi$.
Since elements in the image of $\Phi_{\ell^N}$ can be lifted to $H$, we are done.
\end{proof}

\bibliographystyle{amsalpha}
\bibliography{refs.bib}

\end{document}